\documentclass[10 pt]{amsart}
\usepackage[utf8]{inputenc}

\usepackage{mystyle}
\usepackage{amsmath,amsthm,amssymb,mathtools}
\usepackage{graphicx}
\usepackage{enumitem}

\usepackage{algorithm}
\usepackage{algpseudocode}
\counterwithin{algorithm}{section}

\usepackage[backend=biber, sorting=nyt, maxnames=100,backref=true, firstinits=true, sortcites=true]{biblatex}
\newtheorem*{theorem*}{Theorem}
\newtheorem*{corollary*}{Corollary}
\newtheorem*{lemma*}{Lemma}

\newcommand{\expectation}{\mathbb{E}}

\newcommand{\0}{\emptyset}
\newcommand{\set}[1]{\{#1\}}
\newcommand{\domain}[1]{\text{dom}\left(#1\right)}
\newcommand{\N}{{\mathbb{N}}}

\newcommand{\E}{\mathbb{E}}
\newcommand{\dom}{\mathrm{dom}}
\newcommand{\polylog}{\mathrm{polylog}}

\newcommand{\mb}{\mathbf}

\DeclareMathOperator{\keep}{Keep}
\DeclareMathOperator{\uncolor}{Uncolor}
\DeclareMathOperator{\eq}{Eq}

\def\epsilon{\varepsilon}

\newcommand\numberthis{\addtocounter{equation}{1}\tag{\theequation}}

\makeatletter
\newenvironment{breakablealgorithm}
  {
   \begin{center}
     \refstepcounter{algorithm}
     \hrule height.8pt depth0pt \kern2pt
     \renewcommand{\caption}[2][\relax]{
       {\raggedright\textbf{\ALG@name~\thealgorithm} ##2\par}%
       \ifx\relax##1\relax 
         \addcontentsline{loa}{algorithm}{\protect\numberline{\thealgorithm}##2}%
       \else 
         \addcontentsline{loa}{algorithm}{\protect\numberline{\thealgorithm}##1}%
       \fi
       \kern2pt\hrule\kern2pt
     }
  }{
     \kern2pt\hrule\relax
   \end{center}
  }
\makeatother

\allowdisplaybreaks

\title{Toward Vu's conjecture}
\date{}

\author{Peter Bradshaw}
\email{pb38@illinois.edu}

\author{Abhishek Dhawan}
\email{adhawan2@illinois.edu}

\author{Abhishek Methuku}
\email{methuku@illinois.edu}

\author{Michael C. Wigal}
\email{wigal@illinois.edu}

\thanks{Peter Bradshaw received funding from NSF RTG grant DMS-1937241 and an AMS-Simons Travel Grant.
Abhishek Dhawan received funding from NSF RTG grant DMS-1937241.
Abhishek Methuku is supported by the UIUC Campus Research Board Award RB25050. 
Michael C. Wigal received funding from NSF RTG grant DMS-1937241 and an AMS-Simons Travel Grant.}

\begin{document}
\maketitle

\begin{abstract}
    In 2002, Vu conjectured that graphs of maximum degree $\Delta$ and maximum codegree at most $\zeta \Delta$ have chromatic number at most $(\zeta+o(1))\Delta$. Despite its importance, the conjecture has remained widely open. The only direct progress so far has been obtained in the ``dense regime,'' when $\zeta$ is close to $1$, by Hurley, de Verclos, and Kang. 
    
    In this paper we provide the first progress in the sparse regime $\zeta \ll 1$, the case of primary interest to Vu. We show that there exists $\zeta_0 > 0$ such that for all $\zeta \in [\log^{-32}\Delta,\zeta_0]$, the following holds:  if $G$ is a graph with maximum degree $\Delta$ and maximum codegree at most $\zeta \Delta$, then $\chi(G) \leq (\zeta^{1/32} + o(1))\Delta$. We derive this from a more general result that assumes only that the common neighborhood of any $s$ vertices is bounded rather than the codegrees of pairs of vertices. Our more general result also extends to the list coloring setting, which is of independent interest.
\end{abstract}

\section{Introduction}

\subsection{Background on vertex-coloring and independent sets}

A classical result of Ajtai, Komlós, and Szemerédi~\cite{AKS1980, ajtai1981dense} from the 80s shows that every $n$-vertex triangle-free graph with average degree at most $d$ contains an independent set of size at least $\Omega\left((n / d)\log d\right)$. This influential theorem has sparked extensive research over the past four decades with striking applications in areas such as number theory and discrete geometry, among others. In particular, Ajtai, Koml{\'{o}}s, and Szemerédi~\cite{AKS1980} used this result to establish the celebrated bound $R(3, k) = O(k^2 / \log k)$ on Ramsey numbers. A hypergraph analog of this result was later established by Komlós, Pintz, Spencer, and Szemerédi~\cite{KPS82} who used it to disprove Heilbronn's conjecture on the Heilbronn triangle problem, which asks for the minimum area of a triangle formed by any three points out of a set of $n$ points placed in the unit disk.

Shearer~\cite{Sh83, Sh91} showed that every $n$-vertex triangle-free graph with average degree at most $d$ contains an independent set of size at least $(1 - o(1))(n / d)\log d$, improving the result of Ajtai, Koml\'{o}s, and Szemerédi. Ajtai, Erd{\H{o}}s, Koml{\'o}s, and Szemer{\'e}di~\cite{AEKS81} suggested that such a result may still hold for $K_r$-free graphs for any fixed $r$ (and it may even hold more generally for vertex-coloring as discussed later), and they proved the weaker result that $K_r$-free graphs on $n$ vertices of average degree at most $d$ have an independent set of size at least $\Omega((n / d)\log\log d)$. Several years later, a breakthrough of Shearer~\cite{Sh95} in 1995 improved this bound to $\Omega((n / d)\log d / \log \log d)$, which, up to the leading constant factor, is still the best known. Alon~\cite{alon1996independence} proved that the result of Ajtai, Koml{\'{o}}s, and Szemer\'edi for triangle-free graphs holds more generally for graphs where the neighborhood of every vertex has bounded chromatic number.

Nearly all of the above results on the independence number extend naturally to bounds on the chromatic number as well as the list chromatic number. The \emph{list chromatic number} of a graph $G$, denoted by $\chi_\ell(G)$, is the minimum integer $k$ such that if $L$ is an assignment of lists of colors $L(v) \subseteq \mathbb{N}$ to each $v \in V(G)$ satisfying $|L(v)| \ge k$ for all $v \in V(G)$, 
then $G$ admits a proper vertex coloring $\varphi$ with $\varphi(v) \in L(v)$ for every $v \in V(G)$.

In 1995, Kim~\cite{kim1995} proved that every graph with girth at least five and maximum degree at most $\Delta$ has list chromatic number at most $(1 + o(1))\Delta / \log \Delta$. Independently, Johansson~\cite{johansson1996} showed that every triangle-free graph with maximum degree at most $\Delta$ has chromatic number at most $O(\Delta / \log \Delta)$. In 2019, Molloy~\cite{M17} unified and strengthened these two results by improving the leading constant in Johansson’s bound to match that in Kim’s result (see \cite{bernshteyn2019johansson, hurley2021first, davies2020graph, bonamy2022bounding, martinsson2021simplified} for alternate proofs of Molloy's result). It is not known whether a similar bound holds for $K_r$-free graphs for $r \ge 4$. In this direction, the best known bound is due to Johansson~\cite{J96-Kr} who proved that for every fixed $r$, every $K_r$-free graph of maximum degree at most $\Delta$ has list chromatic number at most $O(\frac{\Delta \log \log \Delta}{\log \Delta})$ (which generalizes the result of Shearer mentioned earlier). A simple proof of this result was later discovered by Molloy~\cite{M17} (see also \cite{davies2020graph, dhawan2025bounds} for more recent improvements in terms of the hidden constant factor). Alon, Krivelevich, and
Sudakov~\cite{AKS} famously conjectured that the $\log \log \Delta$ factor in this bound can be removed, but this is still wide open.

Alon, Krivelevich, and Sudakov~\cite{AKS} bootstrapped Johansson’s theorem to prove a more general result that applies to all \textit{locally sparse graphs}. More precisely, they proved that if $G$ is a graph with maximum degree at most $\Delta$ such that the neighborhood of any vertex spans at most $\Delta^2 / f$ edges for $f \leq \Delta^2 + 1$, then $\chi(G) = O(\Delta / \log f)$. Using a different method, Vu~\cite{Vu02} generalized this result to list coloring. Davies, Kang, Pirot, and Sereni~\cite{davies2020graph} improved this result by showing that it holds with a leading constant of $1 + o(1)$ as $f\to\infty$ (see also \cite{hurley2021first}). These results provide a bound of $o(\Delta)$ on the chromatic number of graphs of maximum degree $\Delta$ under a strong local sparsity condition (and have a wide range of applications; see, e.g., \cite{vu1999some} and the survey \cite{KangKelly2023nibble} on nibble methods). What if we have a much weaker local sparsity condition? Molloy and Reed~\cite{MolloyReed} showed that for every $\zeta > 0$, the following holds for every sufficiently large $\Delta$.  If $G$ is a graph of maximum degree at most $\Delta$ such that the neighborhood of any vertex spans at most $(1 - \zeta)\binom{\Delta}{2}$ edges, then $\chi(G) \leq (1 - \zeta/e^6)\Delta$. Several years later, this result was improved by Bruhn and Joos~\cite{BJ15} and by Bonamy, Perrett, and Postle~\cite{BPP18}. Finally, Hurley, de Verclos, and Kang~\cite{HdVK20} further improved the bound by proving that 
$\chi(G) \le (1 - \zeta/2 + \zeta^{3/2}/6 + o(1))\Delta$, which gives the correct dependence on $\zeta$ as $\zeta \to 0$. 
They used these ideas to make progress on Reed's $\omega$--$\Delta$--$\chi$ conjecture~\cite{R98} and the Erd\H{o}s--Ne{\v{s}}et{\v{r}}il conjecture~\cite{EN85}, 
as well as to obtain the first direct progress toward Vu's conjecture~\cite{Vu02}.

\subsection{Vu's conjecture}
In a graph $G$, the \emph{codegree} of two vertices $u, v$ is the number of distinct neighbors that are common to both $u$ and $v$. Given $\zeta > 0$, how large can the chromatic number be in a graph with maximum degree at most $\Delta$ and maximum codegree at most $\zeta \Delta$? Because of its connection to Kahn's important result on the list chromatic index of linear hypergraphs~\cite{kahn1996asymptotically} and, in turn, to the Erd\H{o}s--Faber--Lovász conjecture~\cite{kahn1997, erdos1981, EFL2023}, Vu~\cite{Vu02} famously proposed the following bold conjecture in 2002.

\begin{conjecture}[Vu~\cite{Vu02}]\label{conj:vu}
For every $\zeta, \varepsilon > 0$ there exists $\Delta_0$ such that the following holds for all $\Delta \ge \Delta_0$. 
If $G$ is a graph with maximum degree at most $\Delta$ and maximum codegree at most $\zeta \Delta$, then 
\[
  \chi_\ell(G) \le (\zeta + \varepsilon)\Delta.
\]
\end{conjecture}

This conjecture remains wide open even if $\chi_\ell(G)$ is replaced by the ordinary chromatic number $\chi(G)$. 
Moreover, even the much weaker statement that $G$ satisfies $\alpha(G) \ge (1/\zeta - \varepsilon)(n/\Delta)$ is still open and was also conjectured by Vu~\cite{Vu02}. As mentioned earlier, Conjecture~\ref{conj:vu} is a far-reaching generalization of a theorem of Kahn~\cite{kahn1996asymptotically} on the list chromatic index of linear hypergraphs. Vu~\cite[p.~109]{Vu02} emphasized its significance, noting that: 
\emph{“The bound in [Conjecture~\ref{conj:vu}], if true, would be an amazing result. For instance, it would immediately imply a deep theorem of Kahn on the list chromatic index of hypergraphs.”} Kahn's theorem, in turn, implies the asymptotic form of the Erd\H{o}s--Faber--Lov\'asz conjecture~\cite{kahn1997, erdos1981, EFL2023}. Recently, Kelly, K\"{u}hn, and Osthus~\cite{kelly2024special} confirmed a special case of Conjecture~\ref{conj:vu} that also recovers several of its significant consequences.

The only direct progress on Vu's conjecture in full generality is due to Hurley, de Verclos, and Kang~\cite{HdVK20}
They established Conjecture~\ref{conj:vu} (for the chromatic number rather than the list chromatic number) in the ``dense regime,'' when $\zeta > 1 - o(\varepsilon^{2/3})$. 
As noted in~\cite{Vu02,HdVK20}, Vu was primarily interested in the ``sparse regime,'' where $\zeta$ is close to $0$.

In this paper, we provide the first progress towards Vu's conjecture in the sparse regime. 
In fact, we assume only that the \emph{$s$-codegree} is bounded rather than the ordinary codegree; this may be of independent interest in view of potential applications. 
Here, in a graph $G$, the $s$-codegree of vertices $u_1, u_2, \ldots, u_s$ is defined as the number of common neighbors of all $u_i$.

\begin{theorem}\label{thm:weak-vu}
For every $s \geq 2$ and $\varepsilon > 0$, there exist $\Delta_0 \in \mathbb{N}$ and $\zeta_0 > 0$ such that the following holds for all $\Delta \ge \Delta_0$ and $\zeta \in [\log^{-16s}\Delta, \zeta_0]$. 
If $G$ is a graph of maximum degree at most $\Delta$ and maximum $s$-codegree at most $\zeta \Delta$, then 
\[
  \chi(G) \le (1 + \varepsilon)\,\zeta^{1/16s}\,\Delta.
\]
\end{theorem}

We remark that the constant $16$ used here (and elsewhere) could be improved slightly; however, we do not optimize it in order to avoid cumbersome calculations. This easily yields the following corollary (see Section~\ref{sec:corollaries} for the proof). In the next subsection, we give a sketch of our proof techniques.

\begin{corollary}
\label{cor:Vu}
For every $s \geq 2$, there exists a constant $C > 0$ such that the following holds for all 
$\Delta \geq 2$ and
$\zeta \in [\log^{-16s} \Delta,1]$. If $G$ is  graph of maximum degree at most $\Delta$ and $s$-codegree at most $\zeta \Delta$, then
\[
  \chi(G) \le C \,\zeta^{1/16s}\,\Delta.
\]
\end{corollary}

\medskip

For graphs with small codegree, the bounds given by Theorem~\ref{thm:weak-vu} and Corollary~\ref{cor:Vu} 
asymptotically improve upon the one implied by the classical result of Alon, Krivelevich and Sudakov~\cite{AKS} for locally sparse graphs discussed earlier. 
Indeed, if a graph $G$ of maximum degree $\Delta$ has codegree at most $\zeta \Delta$, then the neighborhood of any vertex in $G$ has maximum degree at most $\zeta \Delta$
and hence contains at most $O(\zeta^2 \Delta^2)$ edges. 
Thus, the result of Alon, Krivelevich and Sudakov yields the bound $\chi(G) = O(\Delta / \log \zeta^{-1})$, 
whereas Corollary~\ref{cor:Vu} gives $\chi(G) = O(\zeta^{1/16s}\,\Delta)$. 
Since $\zeta^{1/16s} \ll 1/\log(\zeta^{-1})$ for small $\zeta$, our results provide a genuine improvement in this regime.

Our methods also recover the key lemma underlying the recent breakthrough of Campos, Jenssen, Michelen and Sahasrabudhe~\cite{campos2023new} on sphere packings. They proved that there exists a packing of identical spheres in $\mathbb{R}^d$ with density at least $(1-o(1))\,\tfrac{d\log d}{2^{d+1}}$ as $d\to\infty$, improving Rogers’ 1947 bound by a factor of order $\log d$. A central ingredient in their proof is the following graph-theoretic lemma: if $G$ is a graph on $n$ vertices with maximum degree at most $\Delta$ and maximum codegree at most $O(\Delta/(\log \Delta)^7)$, then $\alpha(G)\ge (1-o(1))\,\tfrac{n\log \Delta}{\Delta}$. Our Theorem~\ref{thm:weak-vu} recovers this lemma as a special case (up to a small logarithmic factor in the codegree assumption): setting $s=2$ and $\zeta=\log^{-16s}\Delta=\log^{-32}\Delta$ yields $\chi(G)\le (1+\varepsilon)\,\zeta^{1/16s}\Delta=(1+\varepsilon)\,\Delta/\log \Delta$, which in turn implies $\alpha(G)\ge (1-o(1))\,\tfrac{n\log \Delta}{\Delta}$.

In the setting of list coloring, we prove the following result. 
Theorem~\ref{thm:weak-vu} actually follows by randomly partitioning $G$ and applying this result to each part (see Section~\ref{sec:corollaries} for the details).

\begin{theorem}\label{theo: main result}
    For every $s \geq 2$ and $\epsilon > 0$, there exists $\Delta_0 \in \N$ such that the following holds for all $\Delta \geq \Delta_0$.
    If $G$ is a graph with maximum degree at most $\Delta$ and maximum $s$-codegree at most $\Delta/\log^{16s}\Delta$, then 
    \[
      \chi_\ell(G) \le (1 + \varepsilon)\frac{\Delta}{\log \Delta}.
    \]
\end{theorem}

In fact, we prove the above result in the more general \emph{color-degree} regime (see Theorem~\ref{theo: color-degree version of main result}). 
We also remark that while we assume $s$ is a constant, our proof holds for $s = o(\log \Delta / \log\log\Delta)$.
To simplify the exposition, we only consider the $s = \Theta(1)$ regime.

\subsection{Proof overview}\label{subsec: proof overview}
In this section, we will provide an overview of our proof techniques.
It should be understood that the presentation in this section deliberately ignores certain minor technical issues, and so the actual arguments and formal definitions given in the rest of the paper may be slightly different from how they are described here. 
However, the differences do not affect the general conceptual framework underlying our approach.

Before we provide the details of our strategy, we establish a few definitions.
Let $G$ be a graph and let $L \,:\, V(G) \to 2^{\N}$ be a list assignment for the vertices of $G$.
Given a vertex $v \in V(G)$ and a color $c \in L(v)$, the \textit{color-degree} of $c$ at $v$ is $d_L(v, c) = |N_L(v, c)|$, where $N_L(v, c)$ is the set of neighbors $u$ of $v$ for which $c \in L(u)$.
Similarly, given a collection of vertices $v_1, \ldots, v_s$ and a color $c$, the \textit{$s$-color-codegree} of $c$ with respect to $v_1, \ldots, v_s$ is $d_L(v_1, \dots, v_s, c) = \cap_iN_L(v_i, c)$. (For $s=2$, we simply refer to this as the \emph{color-codegree} rather than the $2$-color-codegree.)
As mentioned earlier, Theorem~\ref{theo: main result} holds in the color-degree setting, where the list size is a function of the maximum color-degree as opposed to $\Delta(G)$.
This framework was pioneered by Kahn \cite{kahn1996asymptotically}, Kim \cite{kim1995}, Johansson \cite{johansson1996,J96-Kr}, and Reed \cite{reed1999list}, among others; see \cite{AndersonBernshteynDhawan, dhawan2024palette, alon2021asymmetric, anderson2025coloring, anderson2024coloring, dhawan2023list, cambie2022independent, alon2020palette} for more recent applications.
Our main result in this setting is as follows:

\begin{theorem}\label{theo: color-degree version of main result}
    For all $s \geq 2$ and $\epsilon > 0$, there exists $d_0 \in \N$ such that the following holds for all $d \geq d_0$.
    Let $G = (V, E)$ be a graph, and let $L \,:\, V \to 2^\N$ be a list assignment for $G$ satisfying the following for $d\geq d_0$:
    \begin{enumerate}
        \item $d_L(v, c) \leq d$ for all $v\in V(G)$ and $c\in L(v)$,
        \item $d_L(v_1, \dots, v_s, c) \leq d/\log^{16s}d$ for all distinct $v_1, \dots, v_s \in V(G)$ and $c\in L(v_1)\cap \dots \cap L(v_s)$, and 
        \item $|L(v)| \geq (1 +\epsilon)d/\log d$ for all $v \in V(G)$.
    \end{enumerate}
    Then, $G$ admits a proper $L$-coloring.
\end{theorem}

Clearly, the above result implies Theorem~\ref{theo: main result}.
The remainder of this section is dedicated to a proof overview for the above result.
For simplicity, we will focus on the case $s = 2$.
Let $G$ be a graph,
and let $L$ be a list assignment for $G$ satisfying the assumptions of Theorem~\ref{theo: color-degree version of main result}. 
In particular, the maximum color-degree of $G$ with respect to $L$ is at most $d$,
and the maximum color-codegree is at most $d/\log^{32}d$, while $|L(v)| \geq (1+\epsilon)d/\log d$ for all $v \in V(G)$. 
In order to find a proper $L$-coloring of $G$, we employ a variant of the so-called ``R\"odl Nibble'' method, in which we randomly color a small portion of $V(G)$ and then iteratively repeat the same procedure for the vertices that remain uncolored; see \cite{KangKelly2023nibble} for a survey of applications of this method to (hyper)graph coloring.
In the rest of this subsection, we highlight some of the key features of our methods.

At the heart of our proof is a version of the ``Wasteful Coloring Procedure'' used to construct a proper partial coloring $\phi$.
This procedure was first introduced by Kim \cite{kim1995} and has been used in a number of results since (see, e.g., \cite{Mahdian2000strong, AndersonBernshteynDhawan}; see also \cite[Ch. 12]{MolloyReed} for a textbook treatment).
The procedure proceeds as follows:
\begin{enumerate}\label{pageref}
    \item Each vertex in $G$ is \emph{activated} with some small probability.
    \item Each activated vertex is assigned a color $\phi(v)$ uniformly at random from $L(v)$.
    \item For each vertex $v$, let $L'(v) \coloneqq L(v) \setminus \{\phi(u)\,:\, u \in N(v)\}$.
    \item If an activated vertex's color survives the previous step, it is assigned the color permanently.
\end{enumerate}
Note that the above procedure is wasteful in the sense that we may delete a color $c$ from $L(v)$ when no vertex in $N(v)$ is permanently assigned $c$.
It turns out that this ``wastefulness'' greatly simplifies the analysis of this procedure (see \cite[Chapter 12.2]{MolloyReed} for a more in-depth discussion of the utility of such wastefulness).
Our hope is to show that, with positive probability, $\phi$ has some desirable properties that allow the coloring to be extended to the entire graph $G$.
Before we describe these properties, we make a few definitions.
Let $G'$ be the subgraph of $G$ induced by the uncolored vertices with respect to $\phi$,
and let $L'$ be the list assignment of $G'$ induced by these restricted lists. 
Our problem is to show that, with positive probability, the graph $G'$ is $L'$-colorable.

Define the following parameters:
\[\ell \coloneqq \min_v|L(v)|, \qquad \ell' \coloneqq \min_v|L'(v)|, \qquad d \coloneqq \max_{v, c}d_L(v, c), \qquad d' \coloneqq \max_{v, c}d_{L'}(v, c).\]
The goal is to show that, with positive probability, the ratio $d'/\ell'$ is noticeably smaller than $d/\ell$.
We would then continue applying our procedure iteratively until the ratio of the maximum color-degree to the minimum list size drops below a small constant, after which the coloring can be completed using standard tools (such as Proposition~\ref{prop: final blow} below).
It turns out that this holds locally in expectation, i.e.,
\[\frac{\mathbb{E}[\deg_{L'}(v,c)]}{\mathbb{E}[|L'(v)|]} \,\leq\, \uncolor \, \frac{d}{\ell},\]
where $\uncolor$ is a certain factor strictly less than $1$ (it is defined explicitly in Section~\ref{sec: wcp} but its exact value is unimportant for this overview).
It is enough to show that the random variables $|L'(v)|$ and $d_{L'}(v,c)$ are concentrated around their expected values.
The claim then follows by an application of the Lov\'asz Local Lemma (Theorem~\ref{thm:lovasz_local_lemma}).

For the concentration, one aims to employ Talagrand's inequality, a powerful concentration tool for random variables satisfying a Lipschitz-like constraint.
We note that for Talagrand's inequality to provide a meaningful bound, one requires the Lipschitz parameter $\xi$ to be of smaller order then $\sqrt{\E[Z]}$, where $Z$ is the random variable of interest.
For $|L'(v)|$, we have $\xi = 1$ and concentration follows easily.
For $d_{L'}(v,c)$, $\xi$ is the maximum color-codegree with respect to $L$.
Since this is at most $d/\polylog(d)$, concentration follows by applying Talagrand's inequality in conjunction with a \textit{random partitioning technique} of \cite{AndersonBernshteynDhawan} (we do not actually employ this approach in our proof, so we omit a detailed description).
However, when applying the result iteratively, we need this bound on the color-codegree to hold \textit{at every step}.
If the color-codegree is at most $d^{\Theta(\epsilon)}$, the proof follows (one can extract this argument from \cite{AndersonBernshteynDhawan}).
In order to push the upper bound to $d/\polylog(d)$, one needs to additionally show that the color-codegree is also concentrated through the nibble.
This is the key technical novelty of our approach.

Let $X \coloneqq d_{L}(u, v, c)$ for some $u, v \in V(G)$ and $c \in L(u) \cap L(v)$, and consider the random variable $X' \coloneqq d_{L'}(u, v, c)$.
Unfortunately, Talagrand's inequality fails since $\xi = \Theta(\E[|X'|])$ and the random partitioning technique of \cite{AndersonBernshteynDhawan} requires $\xi \leq \E[|X'|]/\polylog(d)$.
In their recent breakthrough work on sphere packings, Campos, Jenssen, Michelen, and Sahasrabudhe developed a new concentration tool~\cite[Lemma 3.4]{campos2023new} to concentrate the codegree of a pair of vertices through the independent set variant of the nibble.
Roughly speaking, with respect to their application to concentrating codegrees, their concentration tool provides the following: 
\begin{quote}
    Let $H= (A \sqcup B, E)$ be a bipartite graph with codegree at most $|A|$,
    and let $S$ be a $p$-random subset of $B$; then the random variable $|A \setminus N(S)|$ is concentrated.
\end{quote}
Unfortunately, the random variable $|X'|$ cannot be concentrated using this tool.
To see this, note that $X'$ consists of those vertices $w \in N_L(u, c) \cap N_L(v, c)$ such that 
    \begin{enumerate}[(a)]
        \item \label{uncoloredvertices} $w$ is not colored, and
        \item \label{keptccolor} no neighbor of $w$ is activated and assigned $c$.
    \end{enumerate}
In particular, condition \ref{uncoloredvertices} above precludes us from employing their concentration tool.
(Moreover, their concentration tool cannot be applied to concentrate the $s$-codegree even for the construction of an independent set when $s \geq 3$.)

A key component of our analysis is the interplay between a variant of Talagrand’s inequality (specifically adapted to handle \emph{exceptional outcomes}) and a \emph{Tur\'an-type argument} designed to control the vertices of high degree at every step of the nibble process. Concentration tools of this type have been used extensively in recent years \cite{LiPostle, bruhn2018stronger, AndersonBernshteynDhawan, anderson2024coloring, anderson2025coloring}
so we believe that our approach may be of independent interest. For our proof strategy, we rely on an exceptional variant of Talagrand's inequality developed by Delcourt and Postle in \cite{delcourt2022finding}; see Lemma~\ref{lemma:concentration-inequality} and the discussion afterwards.

Our proof of the concentration of $|X'|$ proceeds as follows. 
Rather than attempting to establish concentration of $|X'|$ directly, 
we first obtain concentration for two auxiliary random variables $|X'_1|$ and $|X'_2|$, 
where $X'_1$ is the number of vertices $w \in N_L(u, c) \cap N_L(v, c)$ satisfying \ref{uncoloredvertices}, 
and $X'_2$ is the number of vertices $w \in N_L(u, c) \cap N_L(v, c)$ satisfying \ref{uncoloredvertices} but not \ref{keptccolor}. 
Since $|X'| = |X'_1| - |X'_2|$, this reduction allows us to deduce concentration of $|X'|$. The following conditions are crucial for bounding the parameter $\beta$ in Lemma~\ref{lemma:concentration-inequality}, the variant of Talagrand’s inequality that we apply; $\beta$ captures a Lipschitz-type property of the random variable.

\begin{enumerate}[(i)]
    \item\label{cond: lipschitz} 
    Let $H$ denote the set of vertices $w \in V(G)$ such that $c \in L(w)$ and $N(w) \cap X$ is ``large.'' 
    Using the K\H{o}v\'ari--S\'os--Tur\'an theorem (Lemma~\ref{lem:Zarankiewicz}) together with the codegree bound, 
    we show that $H$ is ``small,'' allowing us to \textit{ignore} its contribution to $X'$. 

    \item\label{cond: certificate} 
    Let $\Omega^*$ be the exceptional event that there exists a color $c'$ such that ``many'' vertices in $X$ are activated and assigned the color $c'$. We may assume that this exceptional event does not occur by showing that it has only a small probability of occurring and then applying the exceptional variant of Talagrand’s inequality.
\end{enumerate}

In particular, \ref{cond: lipschitz} plays a key role in establishing concentration of $|X'_2|$, 
while the condition that the exceptional event $\Omega^*$ in \ref{cond: certificate} does not occur 
is essential for concentrating both $|X'_1|$ and $|X'_2|$.

We remark that our definition of exceptional events is unusual in the small palette setting ($o(\Delta)$ colors).
Indeed, in all prior such arguments in this setting, the exceptional outcome is designed to control the influence of parameters \textit{outside} of the set of interest (see, e.g., \cite{AndersonBernshteynDhawan, LiPostle, anderson2025coloring}; see also \cite{kelly2024special} for a similar approach in the large palette setting); whereas our definition involves the interaction of variables \textit{inside} the set of interest (similar to the papers \cite{BJ15, HdVK20}, which consider palettes of size $\Theta(\Delta)$).

We conclude with the following observation that allows us to extend the arguments of this section to the $s$-color-codegree: the codegree condition appears only in the proof of condition \ref{cond: lipschitz} and so one can adapt the application of Lemma~\ref{lem:Zarankiewicz} to handle any (constant) $s \geq 2$.

\subsection{Algorithmic considerations}\label{subsection: algorithms}
We establish Theorem~\ref{theo: color-degree version of main result} by iteratively applying the Lov\'asz Local Lemma.
A celebrated work of Moser and Tardos \cite{MT} develops an algorithmic version of the local lemma.
As a result, it is routine to verify that our arguments yield efficient randomized coloring algorithms.
For example, given a graph $G$ of maximum degree at most $\Delta$ and maximum $s$-codegree at most $\Delta/\log^{16s}\Delta$, we obtain a randomized algorithm to color $G$ with $(1+o(1))\Delta/\log\Delta$ colors that runs in time $\mathrm{poly}(\Delta)n$.
The nibble approach has algorithmic implications in other models of computation as well.
We highlight a few in this subsection.

\subsubsection*{Distributed algorithms}

We first discuss distributed algorithms for vertex coloring.
Specifically, in the so-called \textsf{LOCAL} model of distributed computation introduced by Linial in \cite{Linial}.
In the \textsf{LOCAL} model, 
an $n$-vertex graph $G$ abstracts a communication network where each vertex plays the role of a processor and edges represent communication links. 
The computation proceeds in \textit{rounds}. 
During each round, the vertices first perform some local computations and then synchronously broadcast messages to all their neighbors. 
There are no restrictions on the complexity of the local computations and the length of the messages. 
After a number of rounds, every vertex must generate its own part of the output of the algorithm. 
For example, if the goal of the algorithm is to find a proper coloring of $G$, then each vertex must eventually decide on its color.
The efficiency of a \textsf{LOCAL} algorithm is measured by the number of communication rounds required to produce the output.
We focus on the \textit{randomized} version of the \textsf{LOCAL} model, where vertices are distinguished by random labels and the algorithm must yield a correct solution to the problem with probability at least $1 - 1/\mathsf{poly}(n)$.

It is clear that a single iteration of the wasteful coloring procedure can be performed in $O(1)$ rounds of the \textsf{LOCAL} model.
Grable and Panconesi first made this connection in \cite{grable2000fast}, where they designed an $O(\log n)$-round \textsf{LOCAL} algorithm for $\left(\frac{\Delta}{\epsilon\log\Delta}\right)$-coloring graphs of girth at least $5$ satisfying $\Delta \geq \log^{1+\epsilon} n$.
Pettie and Su improved upon this result in a number of ways \cite{PS15}.
They designed a $\log^{1+o(1)}n$-time \textsf{LOCAL} algorithm for $\left(\frac{(4+o(1))\Delta}{\log\Delta}\right)$-coloring $K_3$-free graphs of maximum degree $\Delta = \Omega(1)$.
Additionally, for $\Delta$ sufficiently large, the algorithm has running time $O(\log\Delta + \log^\ast n)$, where $\log^\ast n$ denotes the iterated logarithm of $n$, i.e., the number of times the logarithm function must be applied to $n$ before the result becomes at most $1$. Furthermore, for graphs of girth at least $5$, their algorithm produces a $\left(\frac{(1+o(1))\Delta}{\log\Delta}\right)$-coloring.
There have since been a number of such algorithms inspired by the nibble method for both vertex and edge coloring (cf. \cite{bhattacharya2021online, bhattacharya2024nibbling, chung2014distributed}). As a consequence of our proof techniques, we obtain the following result:

\begin{theorem}\label{theorem: dist}
    For any $\epsilon > 0$ and $s\geq 2$, there exists $\Delta_0 \in \N$ with the following property.
    Let $\Delta\geq \Delta_0$, $n \in \N$, and let $G = (V, E)$ be an $n$-vertex graph of maximum degree at most $\Delta$ and maximum $s$-codegree at most $\Delta/\log^{16s}\Delta$.
    There exists a randomized \textsf{LOCAL} algorithm that computes a $\left(\frac{(1+\epsilon)\Delta}{\log \Delta}\right)$-coloring of $G$ in at most $\log^{1+o(1)}n$ rounds.
    If, additionally, we have $\Delta \geq \log^{\Omega\left(\frac{1}{\epsilon}\right)}n$, the algorithm terminates in $O(\log \Delta + \log^\ast n)$ rounds.
\end{theorem}

The proof of the above result is identical to that in \cite[Section 3.3]{PS15}, \textit{mutatis mutandis}, so we omit it here.

\subsubsection*{Sublinear algorithms}
As mentioned earlier, we obtain a $\mathrm{poly}(\Delta)n$-time randomized algorithm for our problem by combining our proof with the algorithm of \cite{MT}.
When processing dense graphs, this algorithm can be computationally prohibitive. 
This is due to various limitations which arise in processing massive graphs,
such as being required to process the graph in a streaming fashion
(either on a single machine or in parallel across multiple machines due to storage limitations), or simply not having enough time to read the entire input.
This motivates the design of \textit{sublinear algorithms}, i.e., algorithms which require computational resources that are substantially smaller than the size of their input.

In \cite{assadi2019sublinear}, Assadi, Chen, and Khanna introduced a general approach for designing such algorithms based on a technique they called \textit{palette sparsification}. 
The idea of their method, in brief, is as follows. 
Suppose we are trying to properly $q$-color a graph $G$, but, due to limited computational resources, we cannot keep track of all the edges of $G$ at once. 
Let us independently sample, for each vertex $v \in V(G)$, a random set $L(v)$ of colors of size $|L(v)| = \ell \ll q$. 
Define
\[
    E' \,\coloneqq\, \set{uv \in E(G) \,:\, L(u) \cap L(v) \neq \0}.
\]
If we color $G$ by assigning to every vertex $v \in V(G)$ a color from the corresponding set $L(v)$, then only the edges in $E'$ may become monochromatic, so instead of working with the entire edge set of $G$, we only need to keep track of the edges in the (potentially much smaller) set $E'$. 
For this strategy to succeed, we must ensure that, with high probability, it is indeed possible to properly color $G$ using the colors from the sets $L(v)$.
Such a result is referred to as a \textit{palette sparsification theorem}, which find wide-spread applications in the theory of sublinear algorithms; see for example \cite{Sparsification1, alon2020palette, Sparsification2, Sparsification3, dhawan2024palette}. 

It is by now well-established that list-coloring results in the color-degree setting immediately yield palette sparsification results.
As a corollary to Theorem~\ref{theo: color-degree version of main result}, we obtain the following.

\begin{theorem}\label{theorem: palette}
    Let $\epsilon > 0$ be arbitrary.
    There exists a constant $C > 0$ such that for all $\gamma \in (0, 1)$ and $s\geq 2$, there is $\Delta_0 \in \N$ with the following property.
    Let $\Delta\geq \Delta_0$, $n \in \N$, and let $q \geq \ell$ be such that
    \[q \coloneqq \dfrac{(1+\epsilon)\Delta}{\gamma\,\log \Delta}, \qquad \text{and} \qquad \ell \geq \Delta^\gamma + C\sqrt{\log n}.\]
    Let $G = (V, E)$ be an $n$-vertex graph of maximum degree at most $\Delta$ and maximum $s$-codegree at most $\Delta^\gamma/\log^{16s}\left(\Delta^\gamma\right)$.
    Suppose for each vertex $v \in V$, we independently sample a set $L'(v) \subseteq [q]$ of size $\ell$ uniformly at random.
    Then, with probability at least $1 - 1/n$ there exists a proper coloring $\phi$ of $G$ such that $\phi(v) \in L'(v)$ for each $v \in V$.
\end{theorem}

The derivation of the above result from Theorem~\ref{theo: color-degree version of main result} is standard (see, e.g., \cites[Section 3.2]{alon2020palette}[Section 8]{anderson2025coloring}[Section 3]{dhawan2024palette}), so we omit it here. 

\subsection{Structure of the paper.} The rest of the paper is organized as follows.
In Section~\ref{sec:prelim}, we collect some preliminary facts and the main tools we need.
In Section~\ref{sec:corollaries}, we derive Theorem~\ref{thm:weak-vu} and Corollary~\ref{cor:Vu} from Theorem~\ref{theo: main result}.
In Section~\ref{sec: wcp}, we describe the Wasteful Coloring Procedure formally and introduce the key ``nibble lemma'' at the heart of the procedure, which we prove in Section~\ref{sec:nibble-proof}.
Finally, in Section~\ref{section:recursion}, we iteratively apply the nibble lemma to prove Theorem~\ref{theo: color-degree version of main result}.

\section{Preliminaries}\label{sec:prelim}

    In this section, we outline prior known results which are necessary for our proofs. We start by stating the K\H{o}v\'ari--S\'os--Tur\'an theorem~\cite{KovariSosTuran}.
    
    \begin{lemma} [K\H{o}v\'ari--S\'os--Tur\'an theorem \cite{KovariSosTuran}]
\label{lem:Zarankiewicz}
Suppose that $G$ is a bipartite graph with parts $A$ and $B$ of sizes $m$ and $n$, respectively. Suppose that $G$ has no subset $X \subseteq A$ of size $t$ and no subset $Y \subseteq B$ of size $s$ for which $G[X \cup Y]$ is the complete bipartite graph $K_{s,t}$. Then, 
\[e(G) < (t-1)^{1/s} nm^{1-1/s} + (s-1) m.\]
    \end{lemma}

    The following is a straightforward corollary of Lemma~\ref{lem:Zarankiewicz} that allows us to apply K\H{o}v\'ari--S\'os--Tur\'an theorem when the host graph $G$ is not necessarily bipartite. We provide a proof of the corollary for completeness.
    
    \begin{corollary}
\label{cor:KST}
    Let $G$ be a graph with vertex subsets $A$ and $B$ of sizes $m$ and $n$, respectively, for which $A \cup B = V(G)$ and every edge of $G$ has an endpoint in $A$ and an endpoint in $B$. Suppose that $G$ has no complete bipartite subgraph $K_{s,t}$ with parts $X, Y$, where $X \subseteq A$ is of size $t$ and $Y \subseteq B$ is of size $s$. Then,
    \[e(G) < (t-1)^{1/s} nm^{1-1/s} + (s-1) m.\]
\end{corollary}
\begin{proof}
Let $H$ be an auxiliary bipartite graph constructed as follows.
Let $H$ have disjoint partite sets $A'$ and $B'$.
For each vertex $a \in A$, add a vertex $a'$ to $A'$, and for each $b \in B$, add a vertex $b'$ to $B'$.
For each pair $(a,b) \in A \times B$, if $ab \in E(G)$, then add an edge $a'b'$ to $E(H)$.
Since each edge of $G$ corresponds to at least one edge of $H$, clearly $e(H) \geq e(G)$. Furthermore, clearly $H$ is $K_{s,t}$-free.
Therefore, by Lemma~\ref{lem:Zarankiewicz}, we have
\[
e(G) \leq e(H) < (t-1)^{1/s} nm^{1-1/s} + (s-1)m.
\]
This completes the proof of the corollary.
\end{proof}

Next, as mentioned in Section~\ref{subsec: proof overview}, we need a concentration inequality that can handle exceptional outcomes.

\begin{lemma}[Exceptional Talagrand's inequality~\cite{delcourt2022finding}]
\label{lemma:concentration-inequality}
Let $X_1, \dots, X_m$ be a set of independent random variables (often called trials), and let
$(\Omega, \Sigma, \Pr)$ be the product space
that describes all random outcomes of $X_1, \dots, X_m$.
Let $\Omega^* \subseteq \Omega$ be some set of \emph{exceptional} outcomes.
Let $R_1, \ldots, R_n$ be indicator random variables, each defined as a function of the trials $X_1, \ldots, X_m$, and let $R = R_1 + \dots + R_n$.
Suppose that for each non-exceptional outcome $\mb x = (x_1, \dots, x_m) \in \Omega \setminus \Omega^*$ of the trials, the following hold for some $\beta, r > 0$:
\begin{itemize}
    \item For each $i \in [n]$, if the indicator variable $R_i$ satisfies $R_i(\mb x) = 1$, then there is a witness set $W_i(\mb x) \subseteq \{X_1, \dots, X_m\}$ such that the following holds: If another outcome $\mb y = (y_1, \dots, y_m) \in \Omega \setminus \Omega^*$ satisfies $y_j = x_j$ for each $X_j \in W_i(\mb x)$, then $R_i(\mb y) = 1$. If $R_i(\mb x) = 0$, then let $W_i(\mb x) = \emptyset$. 
    \item Each $X_j \in \{X_1, \dots, X_m\}$ is a witness of at most $\beta$ indicator variables $R_i(\mb x)$; that is, $|\{i:X_j \in W_i(\mb x)\}| \leq \beta$.
    \item $ |W_i(\mb x)| \leq r$ for each $1 \leq i \leq n$.
\end{itemize}
Then, for all $\tau  \geq 96 \sqrt{r \beta n} + 128 r \beta +  8 \Pr(\Omega^*) n$, 
\[\Pr \left (|R - \E[R]| \geq    \tau  \right ) < 4 \exp \left ( - \frac{\tau^2}{40 \beta r n} \right ) + 4\Pr(\Omega^*).\]
\end{lemma}

Lemma \ref{lemma:concentration-inequality} easily follows from the concentration inequality \cite[Theorem~3.4]{LiPostle}, which is a special case of the linear Talagrand inequality of Delcourt and Postle \cite[Theorem~4.4]{delcourt2022finding}. The reduction is obtained by defining the verifier function via the witness sets in Lemma~\ref{lemma:concentration-inequality} and observing that, without loss of generality, both $\tau$ and $\mathbb{E}[R]$ are at most $n$.
Additionally, letting $\Omega^* = \emptyset$ above, we obtain a similar concentration inequality to that of Mahdian~\cite{Mahdian2000strong}.
Informally, we apply Lemma \ref{lemma:concentration-inequality} as follows.
Recall the informal description of the Wasteful Coloring Procedure on page~\pageref{pageref}.
For each vertex $v$, let $A_v$ be the random variable such that $A_v = 0$ if $v$ is not activated and $A_v = \phi(v) \in L(v)$ otherwise.
We let $\Omega$ be the product space over the independent random variables $\{A_v\}_{v\in V}$.
The set $\Omega^* \subseteq \Omega$ of exceptional events that we use depends on the random variable that we are interested in concentrating. Typically, we let $\Omega^*$ consist of outcomes in which, for a fixed vertex $v$, there is some color $c$ such that many vertices $u \in N(v)$ satisfy $A_u = c$.
This allows us to control the value of the parameter $\beta$ in Lemma~\ref{lemma:concentration-inequality}.
The exact application of Lemma~\ref{lemma:concentration-inequality} differs very slightly from the above due to the introduction of equalizing coin flips (see Section~\ref{sec: wcp}).

We will also take advantage of the following form of the Chernoff bound, which can be found, for example, in \cite[Chapter 4]{Mitzenmacher}.

\begin{lemma}[Chernoff bound]
\label{lem:chernoff}
Let $Y$ be a random variable that is the sum of pairwise independent indicator variables, and let $\mu = \E[Y]$. Then for any value $\delta > 0$,
\[\Pr(Y > (1 + \delta) \mu) \leq \exp \left( - \frac{ \delta^2 \mu }{2 + \delta }\right).\]
\end{lemma}

We remark that Mitzenmacher \cite{Mitzenmacher} observed it is enough to let $\mu \geq \E[Y]$. Next, we have the symmetric Lov{\'a}sz Local Lemma \cite[Theorem 6.11]{Mitzenmacher}. 
\begin{theorem}[Lov\'asz Local Lemma]
\label{thm:lovasz_local_lemma}
    Let $A_1, A_2, \ldots, A_n$ be events in a probability space. Suppose there exists $p \in [0,1)$ such that for all $1 \le i \le n$ we have $\Pr(A_i) \le p$. Further suppose each $A_i$ is mutually independent from all but at most $d_{LLL}$ other events $A_j$, $j \neq i$ for some $d_{LLL} \in \N$. If $4pd_{LLL} \le 1$, then with positive probability none of the events $A_1, \ldots, A_n$ occur.  
\end{theorem}

For readability, we often omit brackets when writing the expectation of a random variable, with the convention that nonlinear functions take precedence over the expectation operator; that is, we write $\expectation[f(X)] = \expectation f(X)$.
Finally, we often employ the following form of Bernoulli's inequality, along with a related inequality:

\begin{lemma}
\label{lem:bernoulli}
    Let $x \in (0,1)$, and let $r$ be a positive integer. Then, 
    $(1 - x)^r \geq 1 - rx$.
    Furthermore, if $rx < \frac 12$, then 
    \[(1-x)^{-r} \leq \frac 1{1-rx} < 1 + 2rx.\]
\end{lemma}

\section{Proof of Theorem \ref{thm:weak-vu}}
\label{sec:corollaries}

In this section, we derive Theorem~\ref{thm:weak-vu} from Theorem~\ref{theo: main result}.
We also provide a proof of Corollary~\ref{cor:Vu}.
For the reader's convenience, we restate Theorem \ref{thm:weak-vu} here. 

\begin{theorem*}[Restatement of Theorem~\ref{thm:weak-vu}]
    For every $s \geq 2$ and $\varepsilon > 0$, there exist $\Delta_0 \in \mathbb{N}$ and $\zeta_0 > 0$ such that the following holds for all $\Delta \ge \Delta_0$ and $\zeta \in [\log^{-16s}\Delta, \zeta_0]$. 
If $G$ is a graph of maximum degree at most $\Delta$ and maximum $s$-codegree at most $\zeta \Delta$, then $\chi(G) \le (1 + \varepsilon)\zeta^{1/16s} \Delta$.
\end{theorem*}

Fix some $\epsilon > 0$, and assume without loss of generality that $\epsilon < 1/3$.
First, we give a standard partitioning lemma. Similar lemmas appear in \cite{AKS, AlonChoosability, Alon3}.

\begin{lemma}
\label{lem:split}
    There exists $d_{\star}$ such that the following holds for
    all $d \geq d_{\star}$. Let $G$ be a graph of maximum degree $d$ and maximum $s$-codegree at most $t$, where $t \geq \sqrt d$.
    Then, there exists a partition of $V(G)$ into parts $S_1$ and $S_2$ such that for each $i \in \{1,2\}$, $G[S_i]$ has maximum degree at most $\frac 12 d + d^{2/3}$ and maximum $s$-codegree at most $\frac 12 t + t^{2/3}$.
\end{lemma}
\begin{proof}
    For each $v \in V(G)$, we assign $v$ to $S_1$ or $S_2$ uniformly at random.
    Then, to show our desired partition exists, we use the Lov\'asz Local Lemma (Theorem~\ref{thm:lovasz_local_lemma}). For each $v \in V(G)$, let $\mathcal A_v$ be the bad event that $d_{S_1}(v) > \frac 12d + d^{2/3}$ or $d_{S_2}(v) > \frac 12 d + d^{2/3}$,
    where $d_{S_i}(v) = |N(v) \cap S_i|$.
    For each tuple $v_1, \dots, v_s$ with a common neighbor in $G$,
    let
    $\mathcal B_{v_1, \dots, v_s}$ be the bad event that $d_{S_1}(v_1, \dots, v_s) > \frac 12 t + t^{2/3}$ or $d_{S_2}(v_1, \dots, v_s) > \frac 12 t + t^{2/3}$, where  $d_{S_i}(v_1, \dots, v_s)= |\bigcap_{j = 1}^s N(v_j) \cap S_i|$. We show that with positive probability, no bad event occurs.

    Given $v \in V(G)$, we note that for each $i \in \{1,2\}$, $d_{S_i}(v)$ is a binomially distributed variable with an expected value of at most $\frac 12 d$. Therefore, by applying the Chernoff bound (Lemma~\ref{lem:chernoff}) with $\mu = d/2$ and $\delta = 2d^{-1/3}$, 
    \[
        \Pr\left (d_{S_i}(v) > \frac 12 d + d^{2/3} \right ) \leq \exp \left ( - \frac 23 d^{1/3} \right ).
    \]
    Therefore, we may conclude $\Pr(\mathcal A_v) \leq 2\exp( - \frac 23 d^{1/3} )$.

    In a similar manner, given $v_1, \dots, v_s$ and $i \in \{1,2\}$,
    the number of common neighbors of $v_1, \dots, v_s$ in $G[S_i]$ is a binomally distributed variable with an expected value of at most $\frac 12 t$. Therefore, by applying the Chernoff bound (Lemma~\ref{lem:chernoff}) with $\mu = \frac 12 t$ and $\delta = 2t^{-1/3}$,
    \[
        \Pr\left (d_{S_i}(v_1, \dots, v_s) > \frac 12 t + t^{2/3} \right ) \leq \exp \left ( - \frac 23 t^{1/3} \right ) \leq \exp \left ( - \frac 23 d^{1/6} \right ),
    \]
    where we use the fact that $t \geq \sqrt{d}$ in the last inequality.
    Therefore, $\Pr(\mathcal B_{v_1, \dots, v_s} ) \leq 2 \exp \left ( -\frac 23 d^{1/6} \right )$.
    In particular, all bad events considered have probability at most $p \coloneqq 2 \exp \left ( -\frac 23 d^{1/6} \right )$.

    For each event $A_v$, we say $v$ is the \textit{center} of $A_v$.
    Similarly, for each event $B_{v_1, \ldots, v_s}$, we pick an arbitrary $v_i$ to be the center of $B_{v_1, \ldots, v_s}$.
    Let $H$ be the dependence graph of the above set of events. That is,
    \begin{align*}
        V(H) &= \{ A_v : v \in V(G)\} \cup \{ B_{v_1, \dots, v_s} : v_1, \dots, v_s \in V(G)\},
    \end{align*}
    and edges join events which are not mutually independent. 
    We aim to estimate the maximum degree of $H$.
    First, we note that each $v \in V(G)$ is the center of only one event $A_v$ and at most $d^s$ events $B_{v_1, \dots, v_s}$; thus each $v \in V(G)$ is a center of fewer than $2d^{s}$ bad events.
    Furthermore, for each bad event $\mathcal{Z}$ with center $v$,
    the occurrence of $\mathcal{Z}$ depends entirely on the outcomes of vertices in $N(v)$.
    Therefore, if two bad events $\mathcal{Z}$ and $\mathcal{Z}'$ have centers with mutual distance at least $3$, then $\mathcal{Z}$ and $\mathcal{Z}'$ are mutually independent.
    Thus, $\mathcal{Z}$ is mutually independent with all bad events except for at most $2d^{s + 2}$ events $\mathcal{Z}'$ whose centers have distance at most $2$ from $v$.
    Equivalently, the maximum degree of $H$ is at most $d_{LLL} \coloneqq 2d^{s + 2}$.
    As $4pd_{LLL} < 1$ for sufficiently large $d$,
    the Lov\'asz Local Lemma (Theorem~\ref{thm:lovasz_local_lemma}) implies that with positive probability, no bad event occurs, and our desired partition exists.
\end{proof} 

Now, let $d_{\star} \in \mathbb N$ be chosen to 
be large enough for 
Theorem \ref{theo: main result} when $\epsilon$ is replaced by $\epsilon/4$, 
and large enough for Lemma \ref{lem:split}, and also large enough to satisfy finitely many additional inequalities (we note that $d_{\star} \geq \epsilon^{-50s}$ suffices).
Then, let $\Delta_0 = e^{d_{\star}}$, and suppose that $\Delta \geq \Delta_0$.
Let $\zeta_0 = 1/d_{\star}$
and $\log^{-16s} \Delta \leq \zeta \leq \zeta_0$.
Consider a graph $G$
of maximum degree $\Delta $ and maximum $s$-codegree at most $\zeta \Delta $.
Observe that $\zeta \Delta \geq  \Delta / \log^{16s}\Delta $.

We aim to partition $V(G)$ into many parts, so that Theorem \ref{theo: color-degree version of main result} can be applied to each part.
Before partitioning, we define a few parameters.
Let $k$ be the smallest power of $2$ for which 
$\Delta / k\leq  \exp \left( (1 - \frac 14 \epsilon\right)  \zeta^{- \frac 1{16 s}} )$.
We write $i^* = \log_2 k + 1$.
For $i \geq 1$, define:
\begin{enumerate}
    \item $d_1 = \Delta $.
    \item $d_{i+1} = \frac 12 d_i + d_i^{2/3}.$
    \item $t_1 = \zeta \Delta.$
    \item $t_{i+1} = \frac 12 t_i + t_i^{2/3}$.
\end{enumerate}
A simple inductive argument shows that $d_i > t_i$ for each $i \geq 1$.
We also observe that for each $i \geq 1$,
\[
\frac{d_{i+1}}{t_{i+1}} = \frac{d_i}{t_i} \cdot \frac{1 + 2d_i^{-1/3} }{1 + 2t_i^{-1/3}} < \frac{d_i}{t_i},
\]
implying that $ d_i/t_i < d_1/t_1 = \zeta^{-1}$ for all $i \geq 1$.
When $i \leq i^*$, 
we have 
\[d_{i} > \frac{\Delta}{2^{i-1}} > \frac 12 \exp\left( \left(1- \frac 14\epsilon\right) \zeta^{-\frac{1}{16s}}\right) \ge \exp\left( \left(1- \frac 14\epsilon\right) d_{\star}^{\frac{1}{16s}}\right) > d_{\star},\]
where the last inequality follows from our choice of $d_{\star}$. Since $d_i$ is large, the second inequality above implies $\zeta >  (2 \log (2 d_i))^{-16s}  >  1/\sqrt{d_i}$.
With this in hand, we have
\begin{equation}
\label{eq:t_i sqrt}
    t_i > \zeta d_i > \sqrt{d_i}.
\end{equation}

Now, we
partition $V(G)$ as follows.
We initialize a part $S_1 = V(G)$. Then,
for each $i = 1, \dots, i^* - 1 $,
we execute the following:
\begin{quote}
    For each $j \in [2^{i-1}]$, use Lemma \ref{lem:split} to partition $S_j$ into subsets $S_j'$ and $S_j''$
    for which $G[S_j']$ and $G[S_j'']$ have maximum degree at most $d_{i+1}$ and maximum $s$-codegree at most $t_{i+1}$.
    Update $S_j \leftarrow S_j'$, and define $S_{2^{i - 1} + j} \coloneqq S_j''$.
\end{quote}
Since $t_i > \sqrt{d_i}$ for each $i$, our applications of Lemma \ref{lem:split} are valid.
At the end of our procedure, we obtain sets $S_1, \dots, S_k$ that partition $V(G)$ and such that $G[S_j]$ has maximum degree at most $d_{i^*}$ and maximum $s$-codegree at most $t_{i^*}$.
We would like to show that we can use Theorem \ref{theo: color-degree version of main result}
to color each subgraph $G[S_j]$
and then take the union of the colorings to obtain our result.

In our analysis, we often encounter the sum $\sum_{i=1}^{i^* - 1} d_i^{-1/6} $, so we estimate this sum.
Using the fact that $d_i >  d_1/2^{i-1}$ 
for each $i$,
we see that 
\begin{equation*}
\sum_{i=1}^{i^* - 1} d_i^{-1/6} <  d_1^{-1/6} \sum_{i=1}^{i^* - 1} 2^{(i-1)/6} = d_1^{-1/6} \cdot \frac{2^{(i^* - 1)/6} - 1}{\sqrt[6]{2} - 1} = \Delta^{-1/6} \cdot \frac{\sqrt[6]{k}  - 1}{\sqrt[6] 2 - 1} < 9 (k/\Delta)^{1/6}.
\end{equation*}
Since
$\Delta / k> \frac 12 \exp \left ( (1- \frac 14 \epsilon) \zeta^{-\frac{1}{16s}} \right ) > d_{\star}$ is sufficiently large,
\begin{equation}
    \label{eq:sum-UB}
    \sum_{i=1}^{i^* - 1} d_i^{-1/6} < 9d_{\star}^{-1/6} < \frac{1}{16} \epsilon.
\end{equation}

We would like to show that $t_{i^*} < d_{i^*}/\log^{16s} d_{i^*}$; therefore, we estimate $\log d_{i^*}$.
Since $d_{i+1} < \frac 12 d_i (1 + 2 d_i^{-1/6})$,
the inequalities \eqref{eq:sum-UB} and $1 + 2x > e^x$ for $x \in (0,1)$ tell us that
\begin{equation}
\label{eq:i*}
d_{i^*}  < \frac{d_1}{k } \prod_{i=1}^{i^* - 1} (1+ 2 d_i^{-1/6}) < \frac{d_1}{k} \exp \left ( 2 \sum_{i=1}^{i^* - 1} d_i^{-1/6} \right ) < \frac{\Delta}{k } \exp \left ( \frac 18 \epsilon \right ) <  \left (1+\frac 14 \epsilon \right ) \frac {\Delta}k .
\end{equation}
Taking a log and a power, and using the fact that $\log(\Delta / k) \leq (1 - \frac 14 \epsilon) \zeta^{-1/16s}$,
    \begin{align*}
    \log^{16s} d_{i^*} &< \left (  \frac 14 \epsilon + \left (1- \frac 14 \epsilon \right ) \zeta^{-\frac{1}{16s} } \right )^{16s} & \text{using $\log\left(1 + \frac 14 \epsilon\right) < \frac 14 \epsilon$} \\
        & <
 \left (  \left (1- \frac 15 \epsilon \right ) \zeta^{-\frac{1}{16s} } \right )^{16s} &\text{as
 $\zeta^{-1} \geq  \zeta_0^{-1} = d_{\star}$ is sufficiently large} \\
 \numberthis \label{eq:log-UB} &< (1 -  \epsilon) \zeta^{-1} & \text{ as $(1-x)^{32} < 1 - 5x$ for $0 < x < 1/6$}.
\end{align*}
Also, for each $i \geq 1$, \eqref{eq:t_i sqrt} and the inequality $\frac{1}{1+x} > 1-x$ for $x > 0$ tell us that 
\[
\frac{d_{i+1}}{t_{i+1}} = \frac{d_i}{t_i} \cdot \frac{1 + 2 d_i^{-1/3} }{1 + 2 t_i^{-1/3}} > \frac{d_i}{t_i} \left (1 - 2 t_i^{-1/3} \right  ) > \frac{d_i}{t_i} \left (1 - 2 d_i^{-1/6} \right  ).
\]
Therefore, using $1 - \frac 12 x > e^{-x} > 1-x$ for $x \in (0,1)$ along with \eqref{eq:sum-UB} and $d_1/t_1 = \zeta^{-1}$,
\[
\frac{d_{i^*}}{t_{i^*}} > \frac{d_1}{t_1} \prod_{i=1}^{i^* - 1 } \left ( 1 - 2 d_i^{-1/6} \right ) > \zeta^{-1} \exp \left ( - \frac 1{16} \epsilon \right ) > \left (1 - \frac 1{16} \epsilon \right ) \zeta^{-1}.
\]
By \eqref{eq:log-UB},
we thus have $\frac{d_{i^*}}{t_{i^*}} > \log^{16s} d_{i^*}$.

Therefore, by Theorem \ref{theo: color-degree version of main result},
each graph $G[S_j]$ has a proper coloring with
at most $(1 + \frac 14 \epsilon) \frac{d_i^*}{\log d_i^*}$ colors.
Since $d_{i^*} \geq \frac{d_1}{k} > \frac 12 \exp ( (1 - \frac 14 \epsilon) \zeta^{-\frac{1}{16s}}) >\frac 1e \exp ( (1 - \frac 14 \epsilon) \zeta^{-\frac{1}{16s}}) $, 
\begin{align*}
\left (1 + \frac 14 \epsilon \right )\frac{ d_{i^*}}{\log d_{i^*}} &< \left (1 + \frac 14 \epsilon \right ) \frac{ d_{i^*} }{- 1 + (1- \frac 14 \epsilon) \zeta^{-\frac{1}{16s}}}  \\
 &<  \left (1 + \frac 14 \epsilon \right ) \frac{ d_{i^*} }{(1- \frac 13 \epsilon) \zeta^{-\frac{1}{16s}}} &
 \text{as $\zeta^{-1}$ is sufficiently large} \\ 
&< \left (1 + \frac 23 \epsilon \right ) \zeta^{\frac{1}{16s}} d_{i^*}  & \text{ as $\epsilon < \frac 13$} \\
& < \left (1 + \frac 23 \epsilon \right ) \left (1 + \frac 14 \epsilon \right )  \zeta^{\frac{1}{16s}} \frac{\Delta}{k} & \text{by \eqref{eq:i*}} \\
&< (1 + \epsilon) \zeta^{\frac{1}{16s}} \frac{\Delta}{k} & \text{as 
$\epsilon < \frac 13$}.
\end{align*}
Therefore, each $G[S_j]$ has a proper coloring with at most $ (1+ \epsilon) \zeta^{\frac{1}{16s}} \frac{\Delta}{k} $ colors. 
By using a disjoint color set for each of the parts $S_1, \dots, S_k$, we find a proper coloring of $G$ with at most $(1+\epsilon) \zeta^{\frac{1}{16s}} d $ colors. This completes the proof of Theorem \ref{thm:weak-vu}.

Theorem \ref{thm:weak-vu} easily implies Corollary \ref{cor:Vu} as follows.

\begin{corollary*}[Restatement of Corollary \ref{cor:Vu}]
For every $s \geq 2$, there exists a constant $C$ such that the following holds for all $\Delta \geq 2$ and $\zeta \in [\log^{-16s} \Delta,1]$. If $G$ is a graph with maximum degree $\Delta$ and $s$-codegree at most $\zeta \Delta $, then $\chi(G) \leq C \zeta^{1/16s} \Delta $.
\end{corollary*}
\begin{proof}
    Let $\epsilon = 1/4$,
    and let $\Delta_0 \geq 3$ 
    and $\zeta_0$ be chosen to satisfy Theorem \ref{thm:weak-vu}.
    Note that $\Delta_0$ and $\zeta_0$ are absolute constants.
    We claim that the inequality in the corollary holds when $C = 2 \log \Delta_0 \zeta_0^{-1/16s}$.

    Consider a graph $G$ with some maximum degree $\Delta $ whose $s$-codegree is at most $\zeta \Delta $ for some $\zeta \in [\log^{-16s} \Delta, 1]$.
    If $\Delta \geq \Delta_0$ and $\zeta \leq \zeta_0$,
    then by
    Theorem \ref{thm:weak-vu},
    $\chi(G) \leq \frac 54 \zeta^{1/16s} \Delta < C \zeta^{1/16s} \Delta $.
    If $\zeta > \zeta_0$, then, trivially, $\chi(G) \leq 2\Delta \leq C \zeta^{1/16s} \Delta $
    since $\zeta^{1/16s} > \zeta_0^{1/16s}$.
    Finally, suppose that $\Delta < \Delta_0$. Then, $\zeta > \log^{-16 s} \Delta_0$, so that $\zeta^{1/16s} > 1/\log \Delta_0$, and hence $C \zeta^{1/16s}  > 2$.
    Then, $\chi(G) \leq 2\Delta < C \zeta^{1/16s} \Delta $, and the proof is complete.
\end{proof}

\section{The Wasteful Coloring Procedure}\label{sec: wcp}

 Let $G$ be a graph with list assignment $L : V(G) \to \N \setminus \{0\}$. As a remark, we always take our available colors to be positive integers.
 Let $d \in \N$, $\ell \in \N \setminus \{0\}$, and $\eta \in (0,1)$.  Define
\begin{equation}\label{eq:keep_def}
    \keep \coloneqq \keep(d,\ell,\eta) = \left( 1 -  \frac{\eta}{\ell}\right)^{d}, 
\end{equation}
and
\begin{equation}\label{eq:uncolor_def}
    \uncolor \coloneqq \uncolor(d,\ell,\eta) = 1 - \eta \keep.
\end{equation}
For every vertex $v \in V(G)$ and color $c \in L(v)$, define
\begin{equation*}\label{eq:eq_def}
    \eq(v,c) \coloneqq \eq(v,c,d,\ell,\eta) =  \keep \left( 1 - \frac{\eta}{\ell}\right)^{-d_L(v,c)} . 
\end{equation*}

    The usage of these parameters will become apparent after formally introducing the Wasteful Coloring Procedure. For now, note that if $(G,L)$ has maximum color-degree at most $d$ and $\ell \geq 1$, then $\eq(v,c) \in [0,1]$, i.e., $\eq(v,c)$ can be interpreted as a probability. An iteration of the Wasteful Coloring Procedure is as follows.

\vspace{0.1in}
\begin{breakablealgorithm}
    \caption{Wasteful Coloring Procedure}\label{algorithm: wcp}
    \begin{flushleft}
        \textbf{Input}: Two positive integers $d$ and $\ell$, a graph $G$ with a list assignment $L \,:\, V(G) \to 2^{\N\setminus \{0\}}$ with minimum list size at least $\ell$ and maximum color-degree at most $d$, and an activation parameter $\eta \in [0, 1]$. \\
        \smallskip
        \textbf{Output}: A proper partial $L$-coloring $\phi$ of $G$ and list assignment $L' : V(G) \to 2^{\N\setminus \{0\}}$ such that $L'(v) \subseteq L(v) \setminus \set{\phi(u)\,:\, u \in N(v)}$ for all $v \in V(G)$, and the uncolored subgraph $G' \subseteq G$ induced by the vertex set  $V(G) \setminus \domain{\phi}$. 
    \end{flushleft}
    \begin{enumerate}[itemsep = .2cm]
        \item \label{step:trim} \textbf{Trim lists:} For each $v \in V(G)$, arbitrarily discard colors from $L(v)$ so that $|L(v)| = \ell$.

        \item \label{step:prune-edges} \textbf{Prune edges:}
        For each edge $uv \in E(G)$ for which $L(u) \cap L(v) = \emptyset$, delete the edge $uv$ from $E(G)$.
        
        \item \label{step:activate} \textbf{Activate vertices:} Sample a random subset $A \subseteq V(G)$ by including each vertex $v$ independently with probability $\eta$. The vertices in $A$ are considered \emph{activated}.\label{wasteful_color_activation} 
        
        \item \label{step:assignment} \textbf{Assign colors:} For each activated vertex $v \in A$, independently assign a color $c \in L(v)$ uniformly at random.
        
        \item \label{step:update_lists} \textbf{Update lists:} For each activated vertex $v \in A$ with assigned color $c$, remove $c$ from $L(u)$ for each neighbor $u \in N(v,c)$.
        
        \item \textbf{Define $L'$:} For each $v \in V(G)$, let $L'(v)$ be the set of colors remaining in $L(v)$ after the above step.
        
        \item \label{step:last} \textbf{Define $\phi$:} Let $\phi$ be the partial coloring where $\phi(v) = c$ if $v \in A$ was assigned color $c$ and $c \in L'(v)$. For all other vertices, $\phi(v)$ is undefined.

        \item \label{step:coin_flip} \textbf{Equalizing coin flips:} For each  vertex $v \in V(G)$  and  $c \in L(v)$, remove $c$ from $L'(v)$ independently with probability $1 - \eq(v, c)$.

    \end{enumerate}
\end{breakablealgorithm}
\vspace{0.1in}

Note that Step~\ref{step:coin_flip} does not uncolor any vertex $v$ for which $\phi(v)$ is already defined.  We say that a pair $(G,L)$ which has undergone Steps \ref{step:trim} and \ref{step:prune-edges} is \emph{preprocessed}.
In our later proofs, we often work directly with preprocessed pairs $(G,L)$.
After applying the Wasteful Coloring Procedure,
we keep track of the following properties of the output $(G',L',\phi)$:
\begin{itemize}
    \item The minimum list size $|L'(v)|$ of a vertex $v \in V(G')$,
    \item The maximum color-degree $d_{L'}(v,c)$ of a vertex $v \in V(G')$ and $c \in L'(v)$,
    \item The maximum $s$-color-codegree $d_{L'}(v_1, \dots, v_s,c)$ of a distinct $s$-tuple of vertices $v_1, \dots, v_s \in V(G')$ and a color $c \in \bigcap_{i=1}^s L'(v_i)$.
\end{itemize}
Note that while $L'$ is a function on $V(G)$, we are ultimately concerned with its behavior restricted to $V(G')$. The above properties motivate the following technical definition. 
Let $\eta >0$, $s$, $d$, $\ell \in \N \setminus \{0\}$ such that $s \ge 2$. Let $G$ be a graph with list assignment $L$ such that: 
    \begin{itemize}
        \item $4 \eta d < \ell < 8d$,
        \item $\dfrac{1}{\log^2 d} < \eta < \dfrac{1}{ 4 \log d}$,
        \item $|L(v)| = \ell$ for all $v \in V(G)$, 
        \item $d_L(v, c) \leq d$ for all $v\in V(G)$ and $c\in L(v)$,
        \item $d_L(v_1, \dots, v_s, c) \leq d/\log^{16s}d$ for all distinct $v_1, \dots, v_s \in V(G)$ and $c\in L(v_1)\cap \dots \cap  L(v_s)$,
        \item $L(u) \cap L(v) \neq \emptyset$ for each $uv \in E(G)$,
    \end{itemize}
Then,
we say $(G,L)$ together is a $(d,\ell,s ,\eta)$-graph list pair. Note that if $(G,L)$ is a $(d,\ell,s,\eta)$-graph-list pair, then Steps~\ref{step:trim} and \ref{step:prune-edges} of the Wasteful Coloring Procedure are
redundant, i.e., $(G,L)$ is preprocessed. Our main technical lemma is the following.

\begin{lemma}\label{lem:nibble}
    Let $(G,L)$ be a $(d,\ell,s,\eta)$-graph-list pair such that $d$ is at least some sufficiently large constant $\tilde d \in \N$.  Let $\keep$ and $\uncolor$ be defined as in \eqref{eq:keep_def} and \eqref{eq:uncolor_def} respectively. Then there exists a proper partial $L$-coloring $\phi$ and an assignment of subsets $L'(v) \subseteq L(v) \setminus \set{\phi(u)\,:\, u \in N(v)}$ to each $v \in V(G) \setminus \mathrm{dom}(\phi)$ such that, setting
    \[G' \coloneqq G\left[V(G)\setminus \dom(\phi)\right],\]
    the following holds:
    \begin{enumerate}
        \item $|L'(v)| \geq \ell'$ for all $v \in V(G')$,
        \item $d_{L'}(v, c) \leq d'$ for all $v\in V(G')$ and $c\in L'(v)$,
        \item $d_{L'}(v_1, \dots, v_s, c) \leq d'/\log^{16s}d'$ for all distinct $v_1, \dots, v_s \in V(G')$ and $c\in L'(v_1)\cap \dots \cap L'(v_s)$,
    \end{enumerate}
    where
    \[\ell' := \keep\,\ell - \frac{\ell}{\log^{5} \ell}, \qquad \text{and} \qquad d' := \keep\,\uncolor\,d + \frac{d}{\log^{5}  d}.\]
\end{lemma}

In Section \ref{section:recursion}, we will show that Theorem \ref{theo: color-degree version of main result} can be proved through repeated applications of Lemma \ref{lem:nibble}. Before we prove Lemma~\ref{lem:nibble}, we prove a concentration inequality in the next section that plays a key role in our proof.

\section{Proof of Lemma~\ref{lem:nibble}}
\label{sec:nibble-proof}

For each $v \in V(G)$, let $A_v$ be the random variable such that $A_v = 0$ if $v$ is not activated at Step~\ref{step:activate} and $A_v = c \in L(v)$ otherwise (where $c$ is the random color assigned to $v$ at Step~\ref{step:assignment}).
Additionally, for each $v \in V(G)$ and $c\in L(v)$, let $E_{v,c} = 1$ if $c$ was removed from $L(v)$ at Step~\ref{step:coin_flip} and $E_{v,c} = 0$ otherwise.
Throughout this section, we let $\Omega$ consist of the product space over the independent random variables $\{A_v\}_{v\in V} \cup \{E_{v, c}\}_{v\in V, c \in L(v)}$.
Our proof of Lemma~\ref{lem:nibble} proceeds through a series of lemmas. 

\begin{lemma}\label{lem:list_size_expectation}
    Let $(G,L)$ be a $(d,\ell,s,\eta)$-graph-list pair and let $(G',L',\phi)$ be the output of the Wasteful Coloring Procedure. For every $v \in V(G)$ and $c \in L(v)$, we have, 
    \[\Pr\left(c \in L'(v) \right) = \keep, \qquad \text{and} \qquad \expectation |L'(v)| = \ell \cdot \keep,\]
    where $\keep$ is defined as in \eqref{eq:keep_def}.
\end{lemma}

\begin{proof}
    Fix a vertex $v \in V(G)$ and color $c \in L(v)$. 
    Note $c \in L(v)$ is possibly removed from $L(v)$ only in Steps~\ref{step:update_lists} and~\ref{step:coin_flip} of the Wasteful Coloring Procedure. The probability that $c$ remains in $L(v)$ after Step~\ref{step:update_lists}, i.e., no neighbor of $v$ was activated and assigned color $c$,
    is
    \[ \left(1 - \frac{\eta}{\ell}\right)^{d_L(v,c)}.\]
    The probability that $c$ remains after Step~\ref{step:coin_flip} is 
    \[\eq(v,c,\eta) = \keep   \cdot \left(1 - \frac{\eta}{\ell}\right)^{-d_L(v,c)}. \]
    As these events are independent, we may conclude the probability $c \in L'(v)$ is precisely $\keep$. As $|L(v)| = \ell$ for all $v \in V(G)$, by linearity of expectation, the lemma follows.
\end{proof}

For shorthand, we write $\ell'(v) = |L'(v)|$. 
Our next lemma shows that $\ell'(v)$ is concentrated for each $v \in V(G)$.

\begin{lemma}\label{lem:list_size_concentration}
    For sufficiently large $d$, the following holds. Let $(G,L)$ be a $(d,\ell,s, \eta)$-graph list pair and let $(G',L',\phi)$ be the output of the Wasteful Coloring Procedure.
    For all $v \in V(G)$, 
    \[\Pr \left( |\ell'(v) - \expectation[\ell'(v)]| >  \ell/\log^{5} \ell \right) = O(\exp(-\log^2d)).\]
\end{lemma}
\begin{proof}
    Let $v \in V(G)$,
    and let $R := \ell(v) - \ell'(v)$
    denote the random variable for the number of colors removed from $L(v)$ while defining $L'(v)$.
    By Lemma~\ref{lem:list_size_expectation},
    \[ \expectation[R] = (1 - \keep)\ell.\]
    As $\ell > \eta d$, for sufficiently large $d$ the following inequality holds,
    \[ \keep = \left(1 - \frac{\eta}{\ell}\right)^d > \left( 1 - \frac{1}{d}\right)^d > (2e)^{-1}.\]
    For each $u \in N_G(v)$ and $c \in L(v) \subseteq \N \setminus \{0\}$, recall the random variables $A_u$ and $E_{v, c}$ defined at the beginning of the section.
    For each $c \in L(v)$, let $R_c = 1$ if $c \not \in L'(v)$, and $R_c = 0$ otherwise.
    Then we have
    \[ R = \sum_{c \in L(v)} R_c.\]

    We aim to apply Lemma \ref{lemma:concentration-inequality} to concentrate $R$. Therefore, we fix an outcome $\mb x$ of the Wasteful Coloring Procedure, and we construct a witness set for each random variable $R_c$ for which $R_c(\mb x) = 1$.
    If $R_c(\mb x) = 1$, we 
    can construct a witness set $W_c(\mb x)$ of size at most one consisting of either $A_u$ for some $u \in N_G(v)$ or $E_{v,c}$. 
    Note that if $W_c(\mb x) = \{A_u\}$, then $u$ was activated and assigned $c$. Similarly, if $W_c(\mb x) = \{E_{v,c'}\}$, then $c = c'$. Therefore, for each $u \in N_G(v)$, $A_u$ belongs to at most one of the witness sets $W_c(\mb x)$, $c \in L(v)$, and for each $c' \in L(v)$, $E_{v,c'}$ also belongs to at most one of the witness sets $W_c(\mb x)$, $c \in L(v)$.
    Furthermore, as $\eta d < \ell \le 8d$ and $\log^{-2} d < \eta$, for sufficiently large $d$ we may conclude,
    \begin{equation*}
       12\sqrt{\pi \ell} \le 12\sqrt{8 \pi d} < \frac{\eta d}{\log^{5} (8d)} < \frac{\ell}{\log^{5} \ell}.
    \end{equation*} Thus, applying Lemma~\ref{lemma:concentration-inequality} with $\beta = r = 1$, $\Omega^* = \emptyset$, $n = \ell$, and $\tau = \ell/\log^{5}\ell$,
    we have
    \begin{align*}
        \Pr\left( \left|\ell'(v)  - \expectation[\ell'(v)]\right| > \frac{\ell}{\log^{5} \ell} \right) &= \Pr \left(\left| R- \expectation[R] \right| > \frac{\ell}{\log^{5} \ell} \right)\\
        &< 4\exp \left(-\frac{\ell^2 \log^{-10}  \ell}{40\ell} \right) = O(\exp(- \log^2 d)),
    \end{align*}
    where the last equality follows since $d \log^{-2}d \leq \eta d < \ell \le 8d$.
\end{proof}

Our next lemma will be used to estimate the expected values of color-degrees and $s$-color-codegrees in the graph produced by the Wasteful Coloring Procedure.
In this lemma, we specify a small vertex set $H$ whose effect on the list assignment $L'$ produced by the Wasteful Coloring Procedure we effectively ignore.
When applying this lemma, we  let $H$ be a set of vertices with many neighbors in a set whose size we wish to concentrate.
We will see that specifying this vertex set $H$
allows us to concentrate color-degrees and $s$-color-codegrees from above, even when these variables are likely to be far below their means.
We note that a similar approach has been employed in other applications of this variant of the Wasteful Coloring Procedure \cite{AndersonBernshteynDhawan, Mahdian2000strong}.

\begin{lemma}
\label{lem:close_expectations}
    Let $(G,L)$ be a $(d,\ell,s,\eta)$-graph-list pair, and let $(G',L',\phi)$ be the output of the Wasteful Coloring Procedure.
    Let $v \in V(G)$, $c \in L(v)$, and  let $S \subseteq N_L(v,c)$.
    Let $q \geq 0$, and let $H \subseteq V(G)$ be a vertex set 
    such that $|N_L(z,c) \cap H| \leq q$ for each $z \in S$.
    Let $S' \subseteq S$ be the set of vertices $u$ such that after the Wasteful Coloring Procedure,
    \begin{itemize}
        \item $u$ is uncolored, i.e., $u \in V(G')$, 
        \item no vertex in $N_L(u,c) \setminus H$ was activated and assigned $c$, and 
        \item $c$ was not removed from $L'(u)$ by its equalizing coin flip in Step \ref{step:coin_flip}.
    \end{itemize}
    Then,
    \[
        \E |S'| \leq |S| \keep \uncolor +   \frac{ q+1 }{d} |S|,
    \]
    where $\keep$ and $\uncolor$ are defined as in \eqref{eq:keep_def} and \eqref{eq:uncolor_def}, respectively.
\end{lemma}
\begin{proof}
    Fix a choice of $v \in V(G)$ and $c \in L(v)$. 
    For $u \in S$, we estimate $\Pr(u \in S')$ by conditioning on the activation of $u$. We have two cases to consider.

    First, suppose that $u$ is not activated. Then, the conditional probability that $u$ is uncolored is $1$.
    Furthermore,
    the subsequent conditional probability that 
    no vertex of $N_L(u,c) \setminus H$ is activated and assigned $c$, and that 
    $c$ remains in $L(u)$ after its equalizing coin flip,
    is
    \begin{align*}
       (1 - \eta / \ell)^{|N_L(u,c) \setminus H|} \eq(u,c) &= (1 - \eta / \ell)^{d_{L}(u,c) - |N_L(u,c) \cap H|}
    (1 - \eta / \ell)^{d - d_{L}(u,c)}\\
    &= \keep (1 - \eta / \ell)^{ - |N_L(u,c) \cap H|} \\
    &\le \keep(1 - \eta / \ell)^{ - q}  \\
    &<  \keep (1 + 2\eta q / \ell)  < \keep + \frac {q}{d},
    \end{align*}
    where the last two inequalities follow from Lemma \ref{lem:bernoulli} and the fact that $\ell \geq 4 \eta d$ and $\eta q/ \ell \leq  \eta d/ \ell \leq \frac 14 $,
    along with the fact that $\keep < 1$.

    Next, suppose that $u$ is activated. Conditioning on the activation of $u$,
    the following subsequent events are necessary and sufficient for $u \in S'$:
    \begin{enumerate}
        \item[($\mathcal{E}_1$)] $u$ is activated and assigned some color $c' \in L(u)$,
        \item[($\mathcal{E}_2$)] a neighbor $u'$ of $u$ is activated and assigned
        the same color $c'$ as $u$,
        \item[($\mathcal{E}_3$)] no vertex in $N_L(u,c) \setminus H$ is activated and assigned $c$, and
        \item[($\mathcal{E}_4$)] $u$ does not lose $c$ due to its equalizing coin flip. 
    \end{enumerate}
    Hence, we estimate the probability that all events $\mathcal{E}_1$, $\mathcal{E}_2$, $\mathcal{E}_3$, and $\mathcal{E}_4$ occur. Note that the events and probabilities here are implicitly conditioned on the activation of $u$. 

    To aid in our analysis, we partition $\mathcal{E}_1$ into disjoint events $\mathcal{A}_{c'}$ of the form ``$u$ is activated and assigned the color $c'$," so that the disjoint union of the events $\mathcal{A}_{c'}$ over $c' \in L(v)$ is $\mathcal{E}_1$. In particular, we have, 
    \begin{equation}
\label{eqn:partition}
\Pr(\mathcal{E}_1 \cap \mathcal{E}_2 \cap \mathcal{E}_3 \cap \mathcal{E}_4) = \sum_{c' \in L(u) } \Pr(\mathcal{A}_{c'} \cap \mathcal{E}_2 \cap \mathcal{E}_3 \cap \mathcal{E}_4).
    \end{equation}

Thus,  we fix a color $c' \in L(v)$, and we estimate $\Pr(\mathcal A_{c'} \cap \mathcal{E}_2 \cap \mathcal{E}_3 \cap \mathcal{E}_4 ) $.
    We observe that
    \begin{equation}
    \label{eqn:1234}
         \Pr(\mathcal A_{c'} \cap \mathcal{E}_2 \cap \mathcal{E}_3 \cap \mathcal{E}_4) = \Pr(\mathcal A_{c'}) \Pr(\mathcal{E}_2 | \mathcal A_{c'}) \Pr (\mathcal{E}_3 | \mathcal A_{c'}\cap \mathcal{E}_2) \Pr (\mathcal{E}_4), 
    \end{equation}
    where the equality follows from the independence of $\mathcal{E}_4$ with $\mathcal A_{c'} \cap \mathcal{E}_2 \cap \mathcal{E}_3$.
Rather than estimating each of the factors above directly, we take advantage 
    of the following fact:
    \[\Pr(\mathcal{E}_3 ) = \Pr(\mathcal{E}_3 | \mathcal A_{c'}) = \Pr(\mathcal{E}_3 | \mathcal A_{c'} \cap  \mathcal{E}_2) \Pr(\mathcal{E}_2 | \mathcal A_{c'}) + \Pr(\mathcal{E}_3 | \mathcal A_{c'} \cap \overline{\mathcal{E}_2} ) \Pr( \overline{\mathcal{E}_2} | \mathcal A_{c'}),\]
so that 
\[
\Pr(\mathcal{E}_2 | \mathcal A_{c'})
\Pr(\mathcal{E}_3|\mathcal A_{c'} \cap \mathcal{E}_2) = \Pr(\mathcal{E}_3) - \Pr(\mathcal{E}_3| \mathcal A_{c'} \cap \overline{\mathcal{E}_2}) \Pr(\overline{\mathcal{E}_2} | \mathcal A_{c'}) .
\]
By this fact, \eqref{eqn:1234} can be rewritten as 
\begin{equation}
\label{eqn:1234-new}
\Pr(\mathcal A_{c'} \cap \mathcal{E}_2 \cap \mathcal{E}_3 \cap \mathcal{E}_4) = \Pr(\mathcal A_{c'}) (\Pr(\mathcal{E}_3) - \Pr(\mathcal{E}_3| \mathcal A_{c'} \cap \overline{\mathcal{E}_2}) \Pr(\overline{\mathcal{E}_2} | \mathcal A_{c'}))  \Pr(\mathcal{E}_4).
\end{equation}

We estimate the probabilities in \eqref{eqn:1234-new}.
First, $\Pr(\mathcal A_{c'}) = 1/\ell$.
Next, by observation,
\begin{equation}\label{eqn:E3}
     \Pr(\mathcal{E}_3) = (1 - \eta / \ell)^{ |N_L(u,c) \setminus H|}.
\end{equation}

Next, we estimate $\Pr(\mathcal{E}_3 | \mathcal A_{c'} \cap \overline{\mathcal{E}_2} )$.
We consider a neighbor $u' \not \in H$ of $u$. Let $B_c$ be the event that $u'$ is activated and assigned $c$, and let $B_{c'}$ be the event that $u'$ is activated and assigned $c'$.
If $c' = c$, then 
\[\Pr(B_c | \mathcal A_{c'} \cap \overline{\mathcal{E}_2}) = 
\Pr( B_c | \mathcal A_{c'} \cap \overline{B_{c'}}) = 0.\]
If $c' \neq c$, then 
\[\Pr(B_c | \mathcal A_{c'} \cap \overline{\mathcal{E}_2}) = 
\Pr( B_c | \mathcal A_{c'} \cap \overline{B_{c'}}) 
 = \frac{\Pr(B_c \cap \overline{B_{c'}} | \mathcal A_{c'}) }{ \Pr(\overline{B_{c'}}| \mathcal A_{c'}) } = 
  \frac{\Pr(B_c  | \mathcal A_{c'}) }{ \Pr(\overline{B_{c'}}| \mathcal A_{c'}) }  \leq 
  \frac{\eta / \ell}{1 - \eta / \ell} = \frac{\eta}{\ell - \eta}.
\]
Hence, conditioning on $\mathcal A_{c'} \cap \overline{\mathcal{E}_2}$, the conditional probability that $u'$ is activated and assigned $c$ is at most 
$\frac{\eta}{\ell - \eta}$.
Therefore,
\begin{equation}
\label{eqn:e3e1ne2}
\Pr(\mathcal{E}_3| \mathcal A_{c'}  \cap \overline{\mathcal{E}_2}) \geq \left ( 1 - \frac{\eta }{\ell - \eta} \right )^{|N_L(u,c) \setminus H|} \ge  \left ( 1 - \frac{\eta }{\ell - \eta} \right )^{d_L(u,c)}.
\end{equation}
Conditioning on $\mathcal A_{c'}$,
the probability that no neighbor $u'$ of $u$ is assigned $c'$
is 
\begin{equation}
    \label{eqn:ne2e1}
    \Pr(\overline{\mathcal{E}_2} | \mathcal A_{c'}) = (1 - \frac{\eta}{\ell} )^{d_{L}(u,c')  } \geq \keep.
\end{equation}
Finally,
\begin{equation}
\label{eqn:e4}
    \Pr(\mathcal{E}_4) = \eq(u,c) = \left(1 - \frac{\eta}{\ell}\right)^{d - d_{L}(u,c)},
\end{equation}
which gives us all probabilities in \eqref{eqn:1234-new}.

Putting \eqref{eqn:1234-new}, \eqref{eqn:E3}, \eqref{eqn:e3e1ne2}, \eqref{eqn:ne2e1}, and \eqref{eqn:e4} together, we see that $\Pr(\mathcal A_{c'} \cap \mathcal{E}_2 \cap \mathcal{E}_3 \cap \mathcal{E}_4) $ is at most
\begin{equation}
\label{eqn:1234-3}
\frac{1}{\ell}\left ((1 - \frac{\eta}{\ell})^{ |N_L(u,c) \setminus H|} - \left ( 1- \frac{\eta}{\ell - \eta} \right )^{d_L(u,c) }\keep  \right )   \left(1 - \frac{\eta}{\ell}\right)^{d - d_L(u,c)}.
\end{equation}
We aim to simplify \eqref{eqn:1234-3}.
We
define $K  :=  1 - \eta^2/(\ell - \eta)^2$, so that 
\[1 - \frac{\eta }{ \ell - \eta }  = K\left (1 - \frac{\eta }{ \ell} \right ).\]
When $d$ (and thus also $\ell$) is sufficiently large, we have
\[
 K^{d_L(u,c)} > K^d >
 \exp \left ( - \frac{2d \eta^2  }{ (\ell - \eta)^2 }      \right )
 >
 \exp \left ( - \frac{3d  \eta^2 }{ \ell^2  }      \right )
 >
1 - \frac{1}{d} ,
\]
where the last inequality follows from the fact that $\ell > 2 \eta d$.
Therefore, by \eqref{eqn:1234-3},
$\ell > 2 \eta d$,
and the 
inequality $(1 - \eta / \ell)^{ - q}  < (1 + 2\eta q / \ell) $,
\begin{align*}
&~\Pr(\mathcal A_{c'} \cap \mathcal{E}_2 \cap \mathcal{E}_3 \cap \mathcal{E}_4) \\
&< \frac{1}{\ell} \left ((1 - \eta / \ell)^{ |N_L(u,c) \setminus H|} - \left (K \left ( 1- \frac{\eta}{\ell } \right ) \right )^{d_L(u,c) }\keep  \right )   \left(1 - \frac{\eta}{\ell}\right)^{d - d_L(u,c)} \\
 &=   \frac{1}{\ell} \left ((1 - \eta / \ell)^{ -|N_L(u,c) \cap  H|} - K^{d_L(u,c)} \keep  \right )   \left(1 - \frac{\eta}{\ell}\right)^{d } \\
&<  \frac{1}{\ell} \left (   (1 - \eta / \ell)^{ -q} - \keep \left ( 1 - \frac{1}{d}  \right )  \right ) \keep \\
&<  \frac{1}{\ell} \left (   ( 1 + 2\eta q / \ell) - \keep   +  \frac{ \keep }{d}    \right ) \keep \\
&<  \frac{1}{\ell} \left ( (1 - \keep) \keep + \frac{q +1}{d} \right ),
\end{align*}
where we use that $2\eta d < \ell$ in the last inequality. By \eqref{eqn:partition}, we may conclude that 
\[\Pr(\mathcal{E}_1 \cap \mathcal{E}_2 \cap \mathcal{E}_3 \cap \mathcal{E}_4) < (1 - \keep) \keep + \frac{q+1}{d}.\]

Finally, as $u$ is activated in Step~\ref{step:activate} with probability $\eta$ and otherwise remains unactivated, putting both cases together, we have
\begin{eqnarray*}
    \Pr(u \in S') &<& (1 - \eta ) \left (\keep  + \frac {q}d \right ) + \eta \left (( 1 - \keep ) \keep + \frac{q+1}{d} \right ) \\
    &< & \keep (1 - \eta \keep) + \frac{q + 1}{d} .
\end{eqnarray*}
The lemma then follows by linearity of expectation.
\end{proof}

\begin{lemma}\label{lem:color_degree_concentration}
    Let $(G,L)$ be a $(d,\ell,s,\eta)$-graph-list pair and let $v \in V(G)$,
    and let $(G',L',\phi)$ be the output of the Wasteful Coloring Procedure. Let $S \subseteq N_G(v,c)$ for some vertex $v$ and color $c \in L(v)$ such that $|S| \ge d'/\log^{16 s} d'$, where $d' = \keep \uncolor d + d\log^{-5} d$.
    Let $S' \subseteq S$ be the random set of vertices such that $u \in S'$ if and only if during the Wasteful Coloring Procedure,
        \begin{itemize}
        \item $u$ is uncolored, i.e., $u \in V(G')$, 
        \item no vertex in $N_L(u,c)$ was activated and assigned $c$, and 
        \item $c$ was not removed from $L'(u)$ by
        its equalizing coin flip in Step \ref{step:coin_flip}.
    \end{itemize} Then,
    \begin{equation}\label{eq:lem_tech_main}
        \Pr\left(|S'| \ge \keep \uncolor |S| + \frac{|S|}{2 \log^{5} |S|} \right ) = O(\exp(-\log^{2} d)).
    \end{equation}
\end{lemma}

\begin{proof}   
    Let $H$ denote the set of vertices $w \in V(G)$ for which $c \in L(w)$ and $w$ has at least $|S|/\log^{13} d$ neighbors in $S$.
    We first estimate $|H|$. Consider the graph $H'$
    on $H \cup S$
    consisting of all edges with an endpoint in $H$ and an endpoint in $S$.
    As $(G,L)$ is a $(d,\ell,s, \eta)$-graph list pair, $H'$ is $K_{s,t}$-free,
    where $t := (d/\log^{16 s} d) + 1$.
    In particular,
    there is no $K_{s,t}$ with parts $X, Y$ where $Y \subseteq S$ is of size $s$ and $X \subseteq H$ is of size $t$.
    Applying Corollary~\ref{cor:KST} with $n = |S|$ and $m = |H|$, we have
    \[
        |H| \frac{|S|}{\log^{13} d} \le 2e(H') < 2t^{1/s} \cdot |S| |H|^{1-1/s} + 2(s-1)|H|.
    \]
    If $(s-1)|H| \geq t^{1/s} \cdot |S| |H|^{1-1/s}  $, then $|S| \leq 4(s-1) \log^{13} d \ll \frac{d'}{\log^{16s}d'}$, a contradiction.
    Therefore, 
    \[
        |H| \frac{|S|}{\log^{13} d}  < 4t^{1/s} \cdot |S| |H|^{1-1/s}.
    \]
    Rearranging, 
    \begin{equation}\label{eq:H_upperbound}
        |H| \leq 4^s t \log^{13s} d < 5^s d \log^{-3s} d.
    \end{equation}

    Let $S'' \subseteq N_G(v,c)$ be the random subset such that $u \in S''$ if and only if the following hold after the Wasteful Coloring Procedure:
    \begin{itemize}
        \item $u$ remained uncolored, i.e., $u \in V(G')$,
        \item No vertex $u' \in N_G(u,c) \setminus H$ was activated and assigned $c$, and 
        \item $L'(u)$ did not lose $c$ due to its equalizing coin flip in Step \ref{step:coin_flip}.
    \end{itemize}
    Applying Lemma \ref{lem:close_expectations} with $q = |H|$, we have the following for sufficiently large $d$:
    \begin{eqnarray*}
        \E|S''| &\leq& \keep\uncolor|S|+ \frac{|S|(|H| + 1)}{d} \\
        &<& \keep\uncolor|S|+ 2 \cdot 5^s |S| \log^{-3s}d  \\
        & < & \keep\uncolor|S|+  \frac{|S|}{4 \log^{5} d},
    \end{eqnarray*}
   using \eqref{eq:H_upperbound} and the fact that $s \ge 2$. Therefore, as $S'  \subseteq S''$,
    \begin{align}
        \Pr \left ( |S'| \geq \keep \uncolor |S| + \frac{|S|}{2 \log^{5} |S|}  \right ) 
         \leq  &\Pr \left ( |S''| \geq \keep \uncolor |S| + \frac{|S|}{ 2 \log^{5} |S|}  \right )\nonumber \\
        \label{eqn:S''UB}
        \leq  &\Pr \left ( |S''| \geq \E|S''| + \frac{|S|}{4 \log^{5} |S|} \right ).
    \end{align}
    
    We now concentrate $|S''|$ from above.   Define \[ S_{\mu} := \{u \in S : u \in V(G')\},\]
    and
    \[ S_{\kappa} = \{u \in S \, : \,\text{no $u' \in N_L(u,c) \setminus H$ was activated and assigned $c$}\}.\]
    Note that $S'' = S_{\mu} \cap S_{\kappa}$ and $|S_{\mu} \cap S_{\kappa}| = |S_{\mu}| - |S_{\mu} \setminus S_{\kappa}|$. In particular, by linearity of expectation we have that \[ \expectation |S''| = \expectation |S_{\mu} \cap S_{\kappa}| = \expectation |S_{\mu}| - \expectation|S_{\mu} \setminus S_{\kappa}|.  \]
    Continuing from \eqref{eqn:S''UB},
    \begin{align*}
        \Pr\left(|S''| \ge \expectation|S''| + \frac{|S|}{4 \log^{5} |S|}\right) &= \Pr\left(|S_{\mu} \cap S_{\kappa}|\ge \expectation|S_{\mu} \cap S_{\kappa}| + \frac{ |S|}{4 \log^{5} |S|}\right)\\
        &= \Pr\left(|S_{\mu}| - |S_{\mu} \setminus S_{\kappa}| \ge  
 \expectation |S_{\mu}| - \expectation |S_{\mu} \setminus S_{\kappa}| +  \frac{|S|}{4 \log^{5} |S|}\right)
    \end{align*} 
    Therefore, to prove the lemma, it suffices to prove the following two claims:
    \begin{equation}\label{eq:lem_tech_1}
        \Pr\left( |S_{\mu}|    \ge \expectation|S_{\mu}| + \frac{ |S|}{8\log^{5} |S|}\right) =  O(\exp(-\log^{2} d))
    \end{equation}
    and
    \begin{equation}\label{eq:lem_tech_2}
        \Pr\left(|S_{\mu} \setminus S_{\kappa}| \le \expectation|S_{\mu} \setminus S_{\kappa}| - \frac{|S|}{8 \log^{5}|S|}\right) = O(\exp(-\log^{2} d)).
    \end{equation}

     We first prove \eqref{eq:lem_tech_1}, 
     our one-sided concentration of $|S_{\mu}|$. 
     For this, we will use Lemma \ref{lemma:concentration-inequality}.
     Let $\Omega^*$ denote the event that more than $2 \log^{16} d$ vertices of $S$ are activated and assigned any common color. 
     This event will play the role of the exceptional outcome set as defined in Lemma~\ref{lemma:concentration-inequality}. For each color $c'$, let $X_{c'}$ be the random variable that counts the number of vertices of $N_L(v,c)$ activated and assigned color $c'$.  For each $c'$,
     \[  \expectation[X_{c'}] \le d_L(v,c) \cdot \eta/\ell < 1,\]
     where this follows as $ d_L(v,c) \le d$ and  $\eta d < \ell$. Using the Chernoff bound (Lemma~\ref{lem:chernoff}) with $\mu = 1$ (since $\expectation[X_{c'}] < 1$) and $\delta = 2\log^{16}d-1$, we have that 
    \[    \Pr\left( X_{c'}  > 2\log^{16} d\right) < \Pr \left(X_{c'} > \left(2\log^{16}d\right) \expectation[X_{c'}]\right)  = O(\exp(-\log^3  d)). \]
    Note that at most $d \ell$ colors appear in the set $\bigcup_{u \in N_L(v,c)} L(u)$. Therefore, by a union bound and the fact that $\ell \leq 8d$,
    \begin{equation}\label{eq:omega^*_bound}
         \Pr(\Omega^*) \le \sum_{c'} \Pr(X_{c'}) \le  d \ell  \cdot O(\exp(-\log^3 d)) =  O(\exp(-\log^2  d)).
    \end{equation}

Recall that   $u \in V(G')$, i.e., $u$ remains uncolored if: 
    \begin{enumerate}
        \item $u$ was not activated, or
        \item $u$ was activated, assigned a color $c'$, and some $w \in N_L(u,c')$ was also activated and assigned $c'$.
    \end{enumerate}  For each $u \in V(G)$ and $c' \in L(u)$, recall the random variables $A_u$ and $E_{u, c'}$ defined at the beginning of the section.
    For every $u \in S$, let $B_u$ be the indicator random variable for the event $u \in V(G')$. Clearly, $|S_{\mu}| = \sum_{u \in S} B_u$. Now,
    consider a fixed non-exceptional outcome $\mb x \in \Omega \setminus \Omega^*$. 
    For each $u \in S$ satisfying $B_u(\mb x) = 1$,
    we construct a witness set $W_u(\mb x)$ equal to one of the following:
    \begin{itemize}
        \item $\{A_u\}$ if $A_u = 0$;
        \item $\{A_u,A_w\}$ if $A_u = A_w = c'$ for some $c' \in L(u)$ and $w \in N_L(u,c')$.
    \end{itemize}
    It is easy to see that whenever $B_u(\mb x) = 1$, 
    one of the two witness sets above certifies that $B_u(\mb x) = 1$.
    As $\mb x \in \Omega \setminus \Omega^*$,
    each $w \in V(G)$ shares an assigned color with at most $2 \log^{16} d$ other
    vertices in $S$. 
    Therefore, $A_w$ appears in at most $2\log^{16} d  + 1 < 3 \log^{16} d $ witness sets. Finally,
    as $d$ is sufficiently large, $d' \log^{-16}d' \le |S| \le d$, and by \eqref{eq:omega^*_bound}, 
    \[ \frac{|S|}{8 \log^{5} |S|} > \frac{d}{12 \log^{16 s + 5} d} > 96\sqrt{(2)(3\log^{16}d)|S|} + 128(2)(3 \log^{16}d) + 8|S|\Pr(\Omega^*). \]
    
    Thus, we apply Lemma~\ref{lemma:concentration-inequality} with $\beta =3 \log^{16} d $, 
    $r = 2$, $n = |S|$, and $\tau = \frac{|S|}{8 \log^{5} |S|}$ to get 
    \begin{eqnarray*}
        & & \Pr\left(  |S_{\mu}| - \expectation|S_{\mu}|  \ge  \frac{|S|}{8\log^{5} |S|}\right) \leq 4 \exp\left( -\frac{|S|^28^{-2}\log^{-10}|S|}{40   \left(3 \log^{16}d\right) (2) |S| }\right)  + 4 \Pr(\Omega^*)\\
        & \le&  4\exp\left( \frac{-|S|}{(10^6 \log^{10}|S|)\,(\log^{16}d)}\right) + O(\exp(-\log^2 d)) = O(\exp(-\log^2 d)).
    \end{eqnarray*}
    In particular,  \eqref{eq:lem_tech_1} holds.

     We now concentrate $|S_{\mu} \setminus S_{\kappa}|$ from below,
     i.e., we show that inequality \eqref{eq:lem_tech_2} holds. 
     We consider a vertex $u \in S_{\mu} \setminus S_{\kappa}$.
     As $u \in S_{\mu}$,
         $u$ has a witness set $W_u(\mb x)$ of size at most $2$ defined above that certifies that $u \in S_{\mu}$.
        Furthermore, as $u \not \in S_{\kappa}$, 
        one of the following cases must occur:
        \begin{itemize}
            \item a neighbor $w \in N_L(u,c) \setminus H$ was activated and assigned color $c$, or
            \item $u$  lost $c$ due to a coin flip in Step~\ref{step:coin_flip}.
        \end{itemize}

    Let $T_u$ be the  random indicator 
    variable that equals $1$ if and only if  $u \in S_{\mu} \setminus S_{\kappa}$. Then $|S_{\mu} \setminus S_{\kappa}| = \sum_{u \in S} T_u$. Once again, we let $\Omega^*$ be the exceptional event such that more than $2 \log^{16} d$ vertices of $S$ are activated with some common color. By \eqref{eq:omega^*_bound}, 
    $\Pr(\Omega^*) = O(\exp(-\log^2 d))$. For every outcome $\mb x \in \Omega \setminus \Omega^*$ and every $u \in S$ such that $T_u(\mb x) = 1$,
    we define a second witness
    set
    $W'_u(\mb x)$
    according to one of the two cases below:
    \begin{itemize}
        \item If there exists $w \in N_L(u,c) \setminus H$ satisfying $A_w = c$, then we
        choose exactly one such $w$ and define $W'_u(\mb x) = \{A_w\}$.
        \item If $L'(u)$ loses $c$ due its equalizing coin flip (so that $E_{u,c} = 1$), then we define $W_u'(\mb x) = \{E_{u,c}\}$.
    \end{itemize}
    It is easy to check that each set $W_u'(\mb x)$ certifies that $u \not \in S_{\kappa}$.
    Therefore, defining $W''_u(\mb x) = W_u(\mb x) \cup W'_u(\mb x)$ for each $u \in S_{\mu } \setminus S_{\kappa}$, 
    $W''_u(\mb x)$ certifies that $u \in S_{\mu} 
    \setminus S_{\kappa}$. Note that $|W''_u(\mb x)| \leq 3$ for all $u \in S_{\mu} 
    \setminus S_{\kappa}$.

    For each $u \in S$, $E_{u,c}$ appears in at most one witness set, namely $W''_u(\mb x)$. 
    For each $w \in V(G)$, we have already argued that $A_w$ appears in at most $2 \log^{16}d + 1$ witness sets $W_u(\mb x)$.
    If $w \in H$, then $A_w$ does not appear in any witness set $W'_u(\mb x)$. If $w \not \in H$, then 
     $w$ has at most $|S|/\log^{13} d$ neighbors in $S$ by definition of $H$, 
    so
    $A_w$ appears in at most $|S|/\log^{13} d + 1$ witness sets $W'_u(\mb x)$.
    In both cases, as $d$ is sufficiently large, $A_w$ appears in at most $2|S|/\log^{13} d$ witness sets $W''_u(\mb x)$. Furthermore, note that for sufficiently large $d$, we have
    \[ \frac{|S|}{8 \log^{5} |S|} > 300 \frac{|S|}{\log^{13 / 2} d}  +  8 \Pr(\Omega^*) |S| \ > 96 \sqrt{\frac{3 \cdot 2|S|^2}{\log^{13}d} } + 128(3) \frac{2|S|}{\log^{13} d} + 8 \Pr(\Omega^*) |S|.\]

    Hence, applying Lemma~\ref{lemma:concentration-inequality} with $\beta = 2|S|/\log^{13} d$, 
    $r = 3$, $n = |S|$,
    and $\tau = \frac{|S|}{8 \log^{5} |S|}$, we obtain 
    \begin{eqnarray*}
        \Pr \left(\big| \, |S_{\mu} \setminus S_{\kappa}| - \expectation|S_{\mu} \setminus S_{\kappa}| \, \big| \ge \frac{|S|}{8 \log^{5} |S|} \right) 
        &\le& 4 \exp \left(\frac{- |S|^2 8^{-2}\log^{-10}|S| }
        {40 \left(2 |S| \log^{-13}d \right) 3|S|} \right) + 4 \Pr(\Omega^*)\\
        &=& O(\exp(-\log^2 d)).
    \end{eqnarray*}
    Therefore, both \eqref{eq:lem_tech_1} and \eqref{eq:lem_tech_2} hold, and the proof is complete.
\end{proof}

Now, we are ready to prove our main technical lemma.

\begin{proof}[Proof of Lemma~\ref{lem:nibble}]
Let $(G,L)$ be a $(d,s,\ell, \eta)$-graph-list pair.
We aim to show that
after applying the Wasteful Coloring Procedure to $(G,L)$,
the inequalities in the lemma hold with positive probability.
Let $\tilde{d}$ be sufficiently large constant such that Lemmas \ref{lem:list_size_expectation}--\ref{lem:color_degree_concentration} all hold for $d \ge \tilde{d}$ and also
so that
$\tilde{d} \le d' \le d$, recalling that $ d' = \keep\,\uncolor\,d + d/ \log^{5} d$. We also suppose $\tilde{d}$ is sufficiently large such that both $x\log^{-5}x$ and $x\log^{-16s - 5}x$ are increasing functions on $[\tilde{d},\infty)$. We
define the following random variables:
\begin{itemize}
    \item 
    For each $v \in V(G)$ and $c \in L(v)$, we let 
    $X_{v,c}$ be the random variable denoting the number of neighbors of $u \in N_L(v,c)$ which remain uncolored and keep $c$ in their list, i.e. $u \in V(G')$ and $c \in L'(u)$,
    \item 
    For each color $c \in \mathbb N \setminus \{0\}$, 
    distinct tuple $v_1, \dots, v_s \in V(G)$ satisfying $c \in \bigcap_{i=1}^s L(v_i) $ and $\bigcap_{i=1}^s N_L(v_i,c) \neq \emptyset$, we let 
    $X_{v_1,\ldots,v_s,c}$ be the random variable denoting the number of common neighbors of $v_1,\ldots,v_s$ which remain uncolored and keep $c$, i.e., $u \in \bigcap_{i = 1}^s N_L(v_i, c)$ such that $u \in V(G')$ and $c \in L'(u)$.  
\end{itemize}

Recall that $\ell' = \keep\,\ell - \ell/\log^{5} \ell$. For each 
random variable $|L'(v)|$, $X_{v,c}$, and $X_{v_1, \dots, v_s,c}$, we define a bad event as follows.
For each bad event, we also define a \emph{center vertex}.
\begin{itemize}
    \item Let $\mathcal{A}_v$ be the event that $|L'(v)| \le \ell'$. We say that $v$ is the center of $\mathcal{A}_v$.
    \item Let $\mathcal{B}_{v,c}$ be the event $|X_{v,c}| \ge d'$. We say that $v$ is the center of $\mathcal{B}_{v,c}$.
    \item Let $\mathcal{C}_{v_1, \dots, v_s,c}$ be the event $|X_{v_1,\dots,v_s,c}| \ge d'/\log^{16s}d'$. 
    We choose an arbitrary $w \in \bigcap_{i=1}^s N(v_i,c) $, and we say that $w$ is the center of $\mathcal{C}_{v_1,\dots,v_s,c}$.
\end{itemize}

Now, we estimate the probability of each bad event.
For each $v \in V(G)$,
by Lemmas~\ref{lem:list_size_expectation}~and~\ref{lem:list_size_concentration}, 
\begin{align*}
    \Pr(\mathcal{A}_v) &= \Pr\left(|L'(v)| \le \keep \ell - \frac{\ell}{\log^{5} \ell} \right) \\
    &= \Pr\left(|L'(v)| \le \expectation[|L'(v)|] - \frac{\ell}{\log^{5} \ell}\right) = O(\exp(-\log^2 d) ).
\end{align*}

Next, 
for each $v \in V(G)$ and $c \in L(v)$, we estimate the probability of  $\mathcal{B}_{v,c}$.
Let $S = N_L(v,c)$. If $|S| < d'$, then $\Pr(\mathcal{B}_{v, c}) = 0$.
Otherwise, we may apply Lemma \ref{lem:color_degree_concentration} to conclude
\[\Pr\left (|X_{v,c}| \geq \keep \uncolor d_L(v,c) + \frac{d_L(v,c)}{2 \log^{5} d_L(v,c)} \right ) = O(\exp(-\log^2 d) ). \]
In particular, as
\[\keep \uncolor d_L(v,c) + \frac{d_L(v,c)}{2 \log^{5} d_L(v,c)} <  \keep \uncolor d + \frac{d}{\log^{5} d} = d',\]
we have $\Pr(\mathcal{B}_{v,c}) = \Pr(|X_{v,c}| \geq d') = O(\exp(-\log^2 d))$.

Finally, consider a distinct $s$-tuple $v_1, \dots, v_s \in V(G)$
with a common color $c \in \bigcap_{i=1}^s L(v_i)$
and for which $S = N_L(v_1, \dots, v_s ,c)$ is nonempty.
If $|S| < d'\log^{-16s}d'$, then $\Pr(\mathcal{C}_{v_1,\dots,v_s,c}) = 0$. 
If not, we may apply Lemma \ref{lem:color_degree_concentration} to conclude
\[\Pr\left (|X_{v_1,\dots,v_s,c}| \geq \keep \uncolor d_L(v_1, \dots, v_s,c) + \frac{d_L(v_1, \dots, v_s,c)}{2 \log^{5} d_L(v_1, \dots, v_s,c)} \right ) = O(\exp(-\log^2 d) ). \]
In particular, as  
\begin{eqnarray*}
& &
\keep \uncolor d_{L}(v_1, \dots, v_s,c) + \frac{d_{L}(v_1, \dots, v_s,c)}{ 2 \log^{5}  d_L(v_1, \dots, v_s,c)} \\ 
& <  &\keep \uncolor \left ( \frac{d}{\log^{16s}d } \right ) + \frac{d}{\log^{5 + 16s} d} \,=\,
\frac{d'}{\log^{16s} d} \,\leq\, \frac{d'}{\log^{16 s} d'},
\end{eqnarray*}
we have 
\begin{eqnarray*}
\Pr(\mathcal{C}_{v_1, \dots, v_s,c}) &=& \Pr\left(|X_{v_1,\dots,v_s,c}| \ge \frac{d'}{\log^{16 s} d'} \right)  \\
&\leq & \Pr\left(|X_{v_1,\dots,v_s,c}| \ge \keep \uncolor d_{L}(v_1, \dots, v_s,c) + \frac{d_{L}(v_1, \dots, v_s,c)}{2 \log^{5} d}\right)  \\
&=& O(\exp(-\log^2 d)),
\end{eqnarray*}
as desired.

Let $H$ be the dependence graph of the above set of events. That is,
\begin{align*}
    V(H) &= \{ \mathcal{A}_v : v \in V(G)\} \cup \{ \mathcal{B}_{v,c} : v \in V(G), c \in L(v)\} \\
    &\qquad \cup \{ \mathcal{C}_{v_1, \dots, v_s,c} : v_1, \dots, v_s \in V(G), c \in L(v_1) \cap  \dots \cap L(v_s)\},
\end{align*}
and edges join events which are not mutually independent. 
We aim to estimate the maximum degree of $H$.
First, we observe that as $(G,L)$ is a $(d,s,\ell,\eta)$-graph-list pair of maximum color-degree at most $d$ and uniform list size $\ell$,
the maximum degree of $G$ is $d \ell < 8d^2$.
Next, we note that each $v \in V(G)$ is the center of 
only one event $\mathcal A_v$,
at most $\ell < 8d$ events $\mathcal B_{v,c}$,
and at most $\ell \binom ds < 8d^{s+1}$ events $\mathcal C_{v_1, \dots, v_s, c}$; 
thus for $d$ sufficiently large, each $v \in V(G)$ is a center of fewer than $9d^{s+1}$ bad events.
Furthermore, for each bad event $\mathcal{Z}$ with center $v$,
the occurrence of $\mathcal{Z}$ depends entirely on the outcomes of Steps \ref{step:activate}, \ref{step:assignment}, and \ref{step:coin_flip}
of the Wasteful Coloring Procedure
at vertices whose distance from $v$ is at most $3$.
Therefore, if two bad events $\mathcal{Z}$ and $\mathcal{Z}'$ have centers with mutual distance at least $7$, then $\mathcal{Z}$ and $\mathcal{Z}'$ are mutually independent.
Thus, $\mathcal{Z}$ is mutually independent with all bad events except for fewer than $(9d^{s+1})(8d^2)^6 = 2^{18}3^2d^{s+13}$ events $\mathcal{Z}'$ whose centers have distance at most $6$ from $v$.
Equivalently, the maximum degree of $H$ is less than $2^{18}3^2d^{s+13}$.

Now, we apply Theorem \ref{thm:lovasz_local_lemma}.
Each of our bad events has probability at most $p = O(\exp(-\log^2 d))$, and each of our bad events is mutually independent with all but fewer than $d_{LLL}  = 2^{18}3^2d^{s+13}$ other bad events.
As $4pd_{LLL} = o(1) < 1$,
there exists an outcome of the Wasteful Coloring Procedure
in which no bad event occurs, and hence in which all inequalities of Lemma \ref{lem:nibble} hold.
This completes the proof.
\end{proof}

\section{Proof of Theorem~\ref{theo: color-degree version of main result}}\label{section:recursion}

In this section we prove Theorem~\ref{theo: color-degree version of main result} by iteratively applying Lemma~\ref{lem:nibble} until we reach a stage where we can apply the following proposition:

\begin{proposition}\label{prop: final blow}
    Let $G = (V, E)$ be a graph with a list assignment $L$ such that $|L(v)| \geq 8d$ for every $v \in V(G)$, where $d \coloneqq \max_{v \in V,\, c\in L(v)}d_L(v, c)$. Then, there exists a proper $L$-coloring of $G$.
\end{proposition}

Proposition~\ref{prop: final blow} is proved in \cite{reed1999list}; see \cite{reed2002asymptotically} for a stronger result. We continue the notation of Theorem~\ref{theo: color-degree version of main result}. 
Let 
\begin{equation}\label{eq:recursion_base_case_def}
    G_1 \coloneqq G, \quad L_1 \coloneqq L, \quad \ell_1 \coloneqq (1+\epsilon)d/\log d,\quad d_1 \coloneqq d,
\end{equation}
where we may assume that $\epsilon$ is sufficiently small and fixed, say, $\epsilon < 0.001$.
For simplicity of exposition, we will avoid taking floors and ceilings.
Define
\[\kappa = \kappa(\varepsilon) \coloneqq (1+\epsilon/2)\log (1+\epsilon/100)= \frac{\epsilon}{100}  + O(\epsilon^2),\] 
and let
\begin{equation}\label{eq:eta_def}
    \eta = \eta(\varepsilon,d) \coloneqq   \kappa/\log d.
\end{equation}
Note that if $d$ is fixed,  $\eta$ is the same each time we apply Lemma~\ref{lem:nibble}. Recursively, define the following parameters:
\begin{align*}
    \keep_i &\coloneqq \left(1 - \frac{\eta}{\ell_i}\right)^{d_i},
    & \uncolor_i &\coloneqq 1 - \eta\, \keep_i,\\
    \ell_{i+1} &\coloneqq \keep_i\, \ell_i- \frac{\ell_i}{\log^{5}\ell_i}, & d_{i+1}&\coloneqq \keep_i\, \uncolor_i\, d_i + \frac{d_i}{\log^{5}d_i}.
\end{align*}

Suppose during the $i$-th iteration, the following conditions hold:
\begin{enumerate}[label=(C\arabic*)]
    \item\label{item: d_i bound} $d_i \geq \tilde d$,
    \item\label{item: ell_i bound} $4 \eta\,d_i < \ell_i < 8d_i$,
    \item\label{item: eta bound} $\dfrac{1}{\log^2 d_i} < \eta < \dfrac{ 1 }{4 \log d_i}$,
\end{enumerate}
where $\tilde{d}$ is defined as in Lemma~\ref{lem:nibble}. Furthermore, suppose that we have a graph $G_i$ and a list assignment $L_i$ for $G_i$ such that:
\begin{enumerate}[resume, label=(C\arabic*)]
    \item\label{item: L_i} $|L_i(v)| \geq \ell_i$ for all $v \in V(G_i)$,
    \item\label{item: d_i} $d_{L_i}(v, c) \leq d_i$ for all $v\in V(G_i)$ and $c\in L_i(v)$,
    \item\label{item: d_i^2} $d_{L_i}(v_1, \dots, v_s, c) \leq \dfrac{d_i}{\log^{16s} d_i}$ for all distinct $v_1,\dots,v_s \in V(G_i)$ and $c\in L_i(v_1) \cap \cdots \cap L_i(v_s)$.
\end{enumerate}

As we can discard colors arbitrarily such that $|L_i(v)| = \ell_i$ for all $v \in V(G_i)$, 
and as we can remove edges $uv \in E(G)$ for which $L(u) \cap L(v) = \emptyset$,
we may apply Lemma~\ref{lem:nibble} to obtain a partial $L_i$-coloring $\phi_i$ of $G_i$ and an assignment of subsets $L'(v) \subseteq L_i(v) \setminus \set{\phi_i(u) \,:\, u \in N_{G_i}(v)}$ to each vertex $v \in V(G_i) \setminus \dom(\phi_i)$ such that, setting
\[
    G_{i+1} \coloneqq G_i[V(G_i) \setminus \dom(\phi_i)] \quad \text{and} \quad L_{i+1}(v) \coloneqq L'(v),
\]
we get that conditions \ref{item: L_i}--\ref{item: d_i^2} hold with $i+1$ in place of $i$.

Note that assuming $d_0$ is sufficiently large in terms of $\epsilon$, conditions \ref{item: d_i bound}--\ref{item: d_i^2} are satisfied for $i = 1$.
Our goal, therefore, is to show that there is some $i^\star\in \N$ such that
\begin{itemize}
    \item for all $1 \leq i < i^\star$, conditions \ref{item: d_i bound}--\ref{item: eta bound} hold, and
    \item we have $\ell_{i^\star} \geq 8d_{i^\star}$.
\end{itemize}
Since conditions \ref{item: L_i}--\ref{item: d_i^2} hold by construction, we will then be able to iteratively apply Lemma~\ref{lem:nibble} $i^\star - 1$ times and then complete the coloring using Proposition~\ref{prop: final blow}.

We begin by showing that the ratio $d_i/\ell_i$ is decreasing for $d_i$, $\ell_i$ sufficiently large.

\begin{lemma}\label{lemma: decreasing ratios}
    Let $i \ge 1$ be an integer. Suppose  for all $1 \le j \le i$ the following hold: $8d_j \ge \ell_j$,  $\eta \geq 6 \log^{-5} d_j$, and $\eta \geq 6\log^{-5} \ell_j$. If $d$ is sufficiently large with respect to $\varepsilon$, then
    \[
        \frac{d_{i+1}}{\ell_{i+1}} \,\leq\, \frac{d_i}{\ell_i}.
    \]
\end{lemma}

\begin{proof}
    
    We first claim that $\eta d_1/\ell_1 < \kappa$. Indeed, we have by definition (see \eqref{eq:recursion_base_case_def} and \eqref{eq:eta_def}), 
    \begin{equation}\label{eq:decreasing_ratios_base_case}
        \frac{\eta d_1}{\ell_1} = \frac{\eta \log d}{(1 + \varepsilon)} =  \frac{\kappa}{(1 + \varepsilon)} < \kappa.
    \end{equation}
    We now claim if $d_i/\ell_i \le d_1/\ell_1$ and $\eta \geq 6 \log^{-5} d_i$, then,
    \begin{equation}\label{eq:decreasing_ratio_tech}
        \keep_i\, \uncolor_i  < \keep_i - \frac{3}{\log^5\ell_i}.
    \end{equation}
    Indeed, 
    \begin{align*}
        \keep_i\, \uncolor_i &= \keep_i\,\left(1 - \eta\, \keep_i\right) &  \\
        &= \keep_i - \eta\, \left(1 - \frac{\eta}{\ell_i}\right)^{2d_i} & \\
        &\leq \keep_i - \eta\,\left(1 - \frac{2 \eta d_i}{\ell_i}\right) & \text{by Lemma \ref{lem:bernoulli}} \\
        &\leq \keep_i - \eta\,\left(1 - 2 \kappa\right) & \text{by \eqref{eq:decreasing_ratios_base_case}}\\
        &\leq \keep_i - \frac{\eta}{2} & \\
        &\leq \keep_i - \frac{3}{\log^{5}\ell_i} & \text{as $\log^5 \ell_i \ge 6/\eta$}.
    \end{align*}
    Finally, if $\eta \geq 6 \log^{-5} d_i$, 
    then as $\eta = \kappa(\varepsilon)/\log d$ and as
    $d$ is sufficiently large with respect to $\varepsilon$,
    it follows that $d_i$ is also sufficiently large; therefore,
    \begin{equation*}\label{eq:decreasing_ratio_tech_2}
        \log^{5}(8d_i) \leq 2\log^{5}d_i.
    \end{equation*}

    We now prove the lemma by induction on $i$.
    To this end, we suppose that $d_i/\ell_i \le d_1/\ell_1$ and \eqref{eq:decreasing_ratio_tech} holds. It follows that
    \begin{align*}
        \frac{d_{i+1}}{\ell_{i+1}} &= \frac{\keep_i\, \uncolor_i\, d_i + d_i/\log^{5}d_i}{\keep_i\, \ell_i - \ell_i/\log^{5}\ell_i} & \\
        &\leq \frac{d_i\,(\keep_i - 3/\log^{5}\ell_i + 1/\log^{5}d_i)}{\ell_i\,(\keep_i - 1/\log^{5}\ell_i)} & \text{by \eqref{eq:decreasing_ratio_tech}}\\
        &\leq \frac{d_i}{\ell_i} & \text{as $\ell_i \leq 8d_i$},
    \end{align*}
    as desired.
\end{proof}

For computational purposes, it is convenient to ignore the error terms $d_i / \log^{5} d_i$ and $\ell_i / \log^{5} \ell_i$.
The following lemma shows that provided both $\ell_i$ and $d_i$ are large, we may indeed ignore these error terms. Such an approach was first employed by Kim \cite{kim1995}; see \cite[Ch. 12]{MolloyReed} for a textbook treatment of the argument and \cite{AndersonBernshteynDhawan} for a more recent application.

\begin{lemma}\label{lemma: error terms irrelevant}
    Let $\hat{\ell}_1 \coloneqq \ell_1$, $\hat{d}_1 \coloneqq d_1$, and recursively define:
    \begin{align*}
        \hat{\ell}_{i+1} &\coloneqq \keep_i\, \hat{\ell}_i, \\ 
        \hat{d}_{i+1} &\coloneqq \keep_i\, \uncolor_i\, \hat{d}_i.
    \end{align*}
    If $\eta \geq 20 \log^{-3}d_j$, $\eta \geq 20 \log^{-3}\ell_j$ and $\ell_j \leq 8d_j$ for all $1 \leq j < i$, and if $d$ is sufficiently large with respect to $\varepsilon$, then
    \begin{itemize}
        \item $|\ell_i - \hat{\ell}_i| \leq \dfrac{\hat{\ell}_i}{\log\hat{\ell}_i}$,
        \item $|d_i - \hat{d}_i| \leq \dfrac{\hat{d}_i}{\log\hat{d}_i}$.
    \end{itemize}
\end{lemma}

\begin{proof}
    Before we proceed with the proofs, let us record a few inequalities. 
    First, as $d$ is sufficiently large, by Lemmas~\ref{lem:bernoulli} and \ref{lemma: decreasing ratios}, we have
    \begin{equation}\label{eq:keep_lower}
        \keep_i  = (1 - \eta/\ell_i)^{d_i} \geq 1 - \frac{\eta d_i}{\ell_i} = 1-  \frac{d_i\, \kappa}{\ell_i\log d} \ge 1-  \frac{d_1\, \kappa}{\ell_1\log d} > 1-\kappa.
    \end{equation}
    Also, as $d_i \geq \ell_i/8$, we have
    \begin{equation}\label{eq:keep_upper}
        \keep_i \leq \exp\left(-\dfrac{\eta\, d_i}{\ell_i}\right) \leq \exp\left(-\dfrac{\eta}{8}\right).
    \end{equation}
    It follows from \eqref{eq:keep_lower} that
    \begin{align}\label{eq:keep_uncolor_lower}
        \keep_i\, \uncolor_i &= \keep_i - \eta\, \keep_i^2 > 1 - \kappa + \eta\, \keep_i^2 \notag\\
        &= 1 - \kappa \left( 1 + \frac{\keep_i^2}{\log d}\right) \geq 1 - 2\kappa.
    \end{align}

    We will now prove the lemma by induction on $i$.
    For the base case $i = 1$, the claim is trivial. 
    Assume now that it holds for some $i \geq 1$,
    and consider $i + 1$.
    Note that $\hat{\ell}_i \geq \ell_i$ and $\hat{d}_i \leq d_i$.
    We have
    \begin{align*}
        \hat{\ell}_{i+1} &= \keep_i\, \hat{\ell}_i & \\
        &\leq \keep_i\,\left(\ell_i + \frac{\hat{\ell}_i}{\log\hat{\ell}_i}\right) & \text{by induction}\\
        &= \ell_{i+1} + \frac{\ell_i}{\log^{5}\ell_i} + \frac{\keep_i\hat{\ell}_{i}}{\log\hat{\ell}_i} & \text{by definition of $\ell_{i+1}$}\\
        &= \ell_{i+1} + \frac{\hat{\ell}_{i+1}}{\log\hat{\ell}_{i+1}} + \frac{\ell_i}{\log^{5}\ell_i} - \keep_i\hat{\ell}_{i}\left(\frac{1}{\log\hat{\ell}_{i+1}} - \frac{1}{\log\hat{\ell}_i}\right). &
    \end{align*}
    It is now enough to show that 
    \begin{align}\label{eq:ell_bound_with_hat}
        \frac{\ell_i}{\log^{5}\ell_i} \leq \keep_i\hat{\ell}_{i}\left(\frac{1}{\log\hat{\ell}_{i+1}} - \frac{1}{\log\hat{\ell}_i}\right).
    \end{align}
    Since the function $f(x) = 1/x$ is convex, we have
    \begin{align*}
        \frac{1}{\log\hat{\ell}_{i+1}} - \frac{1}{\log\hat{\ell}_i} &= f(\log\hat{\ell}_{i+1}) - f(\log\hat{\ell}_{i}) \\
        &\geq f'(\log\hat{\ell}_{i})(\log\hat{\ell}_{i+1} - \log\hat{\ell}_{i}) \\
        &= \frac{-\log \keep_i}{\log^{2}\hat{\ell}_i} \\
        &\geq \frac{\eta}{8\log^{2}\hat{\ell}_i},
    \end{align*}
    where the last inequality follows by \eqref{eq:keep_upper}.
    With this in hand, we have
    \begin{align*}
        \keep_i\hat{\ell}_{i}\left(\frac{1}{\log\hat{\ell}_{i+1}} - \frac{1}{\log\hat{\ell}_i}\right) \ge \frac{\keep_i \eta \hat{\ell}_i }{8 \log^2 \hat{\ell}_i} \geq \frac{(1-\kappa)\eta \hat{\ell}_i}{8\log^{2}\hat{\ell}_i} > \frac{\eta\,\ell_i}{10\log^{2}\ell_i},
    \end{align*}
    where the second inequality is due to \eqref{eq:keep_lower} and the last step follows since $\hat{\ell}_i \geq \ell_i$ and $\varepsilon$ was chosen to be sufficiently small.
    In order to prove \eqref{eq:ell_bound_with_hat}, it is enough to have $\eta \geq 10 \log^{-3} \ell_i$,
    which follows by the lower bound on $\ell_i$ in the statement of the lemma.
    
    Similarly, we have
    \begin{align*}
        \hat{d}_{i+1} &= \keep_i\,\uncolor_i\, \hat{d}_i & \\
        &\geq \keep_i\,\uncolor_i\,\left(d_i - \frac{\hat{d}_i}{\log \hat{d}_i}\right)  & \text{by induction}\\
        &= d_{i+1} - \frac{d_i}{\log^{5}d_i} - \frac{\keep_i\,\uncolor_i\hat{d}_{i}}{\log\hat{d}_i} & \text{by definition of $d_{i+1}$}\\
        &= d_{i+1} - \frac{\hat{d}_{i+1}}{\log\hat{d}_{i+1}} - \frac{d_i}{\log^{5}d_i} + \keep_i\,\uncolor_i\hat{d}_{i}\left(\frac{1}{\log\hat{d}_{i+1}} - \frac{1}{\log\hat{d}_i}\right). &
    \end{align*}
    It is now enough to show that 
    \begin{align}\label{eq:d_bound_with_hat}
        \frac{d_i}{\log^5 d_i} \leq \keep_i\,\uncolor_i\hat{d}_{i}\left(\frac{1}{\log\hat{d}_{i+1}} - \frac{1}{\log \hat{d}_i}\right).
    \end{align}
    By \eqref{eq:keep_upper} and since the function $1/x$ is convex, we have
    \begin{align*}
        \frac{1}{\log\hat{d}_{i+1}} - \frac{1}{\log\hat{d}_i} \geq \frac{-\log (\keep_i\uncolor_i)}{\log^{2}\hat{d}_i} \geq  \frac{-\log \keep_i}{\log^{2}\hat{d}_i} \geq \frac{\eta}{8\log^{2}\hat{d}_i},
    \end{align*}
    where the last inequality follows from \eqref{eq:keep_upper}. With this in hand, we have the following as a result of \eqref{eq:keep_uncolor_lower}:
    \begin{align*}
        \keep_i\,\uncolor_i\hat{d}_{i}\left(\frac{1}{\log\hat{d}_{i+1}} - \frac{1}{\log\hat{d}_i}\right) \geq \frac{\eta\,(1-2\kappa)\hat{d}_i}{8\log^{2}\hat{d}_i} \geq \frac{\eta\,\hat{d}_i}{10\log^{2}\hat{d}_i},
    \end{align*}
    where the second inequality follows from \eqref{eq:keep_uncolor_lower} and the last step follows as $\epsilon$ was chosen to be sufficiently small.

    In order to prove \eqref{eq:d_bound_with_hat}, it is enough to show
    \[\frac{d_i}{\log^{5}d_i} \left( \frac{\hat{d}_i}{\log^{2}\hat{d}_i}\right)^{-1} \leq \frac{\eta}{10}.\]
    To this end, we note that
    \begin{align*}
        \frac{d_i}{\log^{5}d_i} \left( \frac{\hat{d}_i}{\log^{2}\hat{d}_i}\right)^{-1}= \frac{d_i}{\hat{d_i}}\,\frac{\log^{2}\hat{d}_i}{\log^{5}d_i} \leq \left(1 + \frac{1}{\log\hat{d_i}}\right)\frac{1}{\log^{3}d_i} < \frac{2}{\log^{3}d_i},
    \end{align*}
    where we use the induction hypothesis and the fact that $\hat{d}_i \leq d_i$.
    By the lower bound on $d_i$ in the statement of the lemma, we have
    \[\frac{2}{\log^{3}d_i} \leq \frac{\eta}{10},\]
    as desired.
\end{proof}

Next we show that $\ell_i$ never gets too small:

\begin{lemma}\label{lemma:l_lower} 
    If $d$ is sufficiently large with respect to $\varepsilon$, the following holds. Suppose $i \ge 1$ such that for all $j < i$, we have $\ell_j \leq 8d_j$, then $\ell_i \geq d^{\epsilon/15}$.
\end{lemma}

\begin{proof}
    For $i = 1$, by definition of $\ell_1$ (see \eqref{eq:recursion_base_case_def}) 
    \[ \ell_1 = (1 + \varepsilon)d/\log d.\]
    Thus as $d$ is sufficiently large, the claim holds.

    The proof is by induction on $i$. Let $i \geq 1$, and let $r_i \coloneqq d_i/\ell_i$ and $\hat{r}_i \coloneqq \hat{d}_i/\hat{\ell}_i$, where $\hat{d}_i$ and $\hat{\ell}_i$ are defined as in Lemma~\ref{lemma: error terms irrelevant}. Now we assume that the desired bound holds for $\ell_1, \ldots, \ell_i$ and consider $\ell_{i+1}$.
    By the induction hypothesis, we have $\ell_j \geq d^{\epsilon/15}$ for $j \leq i$.
    Furthermore, as we assume $d_j \geq \ell_j / 8$ and $d$ is sufficiently large,  we have $d_j \geq d^{\epsilon/20}$.
    In particular, we may apply Lemmas~\ref{lemma: decreasing ratios} and~\ref{lemma: error terms irrelevant}.
    The following claim will assist with our proof:

    \begin{claim}\label{claim: lb on hat ell and hat d}
        Suppose $i \geq 1$ and that for all $1 \leq j \leq i$, we have $\ell_j \leq 8d_j$ and $\ell_i \geq d^{\epsilon/15}$. Then,
        \[\hat{\ell}_i \geq d^{\epsilon/20}, \qquad \text{ and } \qquad \hat{d}_i \geq d^{\epsilon/20} .\]
    \end{claim}

    \begin{claimproof}
        By definition of $\hat{\ell}_i$ and the induction hypothesis, we have,
        \[ \hat{\ell}_i \ge \ell_i \ge d^{\varepsilon/15}.\]
        Similarly, by Lemma~\ref{lemma: error terms irrelevant}, the induction hypothesis, and since $d_i \geq \ell_i/8$, we have
        \[\hat{d}_i\left(1 + \frac{1}{\log \hat{d}_i}\right) \geq d_i \geq \frac{\ell_i}{8}\implies \hat{d}_i\left(1 + \frac{1}{\log \hat{d}_i}\right) \geq \frac{d^{\epsilon/15}}{8}.\]
        If $\hat{d}_i < d^{\epsilon/20}$, we have
        \[\hat{d}_i\left(1 + \frac{1}{\log \hat{d}_i}\right) < d^{\epsilon/20} + \frac{20d^{\epsilon/20}}{\epsilon\log d} \ll \frac{d^{\epsilon/15}}{8},\]
        a contradiction for $d$ sufficiently large.
    \end{claimproof}

    \smallskip
    
    By Claim~\ref{claim: lb on hat ell and hat d}, we have that
    \[ \left(1 + \frac{1}{\log\hat{d}_i}\right)\left(1 + \frac{2}{\log\hat{\ell}_i}\right) = \left(1 + O\left(\frac{1}{\epsilon\log d}\right)\right),\]
    and as $d$ is sufficiently large, we may conclude
    \begin{equation}\label{eq:l_lower_bound_tech_1}
         \left(1 + \frac{1}{\log\hat{d}_i}\right)\left(1 + \frac{2}{\log\hat{\ell}_i}\right) < \frac{1 + \varepsilon}{1 + \varepsilon/2}.
    \end{equation}
    Furthermore, as $\ell_i \geq d^{\epsilon/15}$, and as $d$ is sufficiently large, 
    \begin{equation}\label{eq:bound1}
        1 - \frac{\eta}{\ell_i} \geq \exp{\left(-\frac{\eta}{(1-\epsilon/4)\ell_i}\right)}.
    \end{equation}
    Note that \[r_1 = \hat{r}_1 = \frac{\log d}{(1+\epsilon)}.\]
    In addition, as $\varepsilon < 1$, \[(1-\epsilon/4)(1+\epsilon) \geq 1 +\epsilon/2.\]
    Hence, by Lemma \ref{lem:bernoulli},
    \begin{align*}
        \keep_i &= \left(1 - \frac{\eta}{\ell_i}\right)^{d_i} & \\
        &\ge 1 -  \frac{\eta d_i}{\ell_i} & \\
        &\geq \exp{\left(-\frac{\eta}{(1-\epsilon/4)}\, r_i\right)} &\text{by \eqref{eq:bound1}}\\
        &\geq \exp{\left(-\frac{\eta}{(1-\epsilon/4)}\, r_1\right)} &\text{by Lemma \ref{lemma: decreasing ratios}}\\
        &\geq \exp{\left(-\frac{\kappa}{(1+\epsilon/2)}\right)}. 
    \end{align*}
    \noindent
    With this bound on $\keep_i$, we can bound $\hat{r}_i$ as follows:
    \begin{align}\label{eq:rhat_ub}
        \hat{r}_i &= \hat{r}_1\prod\limits_{j < i}\uncolor_j \notag\\
        &= \hat{r}_1\prod\limits_{j < i}\left(1 - \eta\, \keep_j \right) \notag\\
        &\leq \frac{\log d}{1+\epsilon}\left(1 - \eta\, \exp{\left(-\frac{\kappa}{(1+\epsilon/2)}\right)}\right)^{i-1}.
    \end{align}

    With these bounds in hand, we obtain the following bound on $r_i$ for $d$ large enough in terms of $\epsilon$:
    \begin{align*}
        r_i &\leq \hat{r}_i \left( 1 + \frac{1}{\log\hat{d}_i}\right) \left( 1 - \frac{1}{\log\hat{\ell}_i}\right)^{-1} & \\
        &\leq \hat{r}_i\left(1 + \frac{1}{\log\hat{d}_i}\right)\left(1 + \frac{2}{\log\hat{\ell}_i}\right) & \text{by Lemma \ref{lem:bernoulli}} \\
          &<  \left ( \frac{1+\epsilon}{1 + \epsilon/2} \right ) \hat r_i & \text{by \eqref{eq:l_lower_bound_tech_1}}\\
        &\leq \frac{\log d}{1+\epsilon/2}\left(1 - \eta\, \exp{\left(-\frac{\kappa}{(1+\epsilon/2)}\right)}\right)^{i-1} & \text{by \eqref{eq:rhat_ub}}.
    \end{align*}

     Note as $\epsilon$ was chosen sufficiently small, $(1-\epsilon/4)(1+\epsilon/2) \geq 1+\epsilon/8$. Applying this and the above bound on $r_i$, we can get a better bound on $\keep_i$:
    \begin{align*}
        \keep_i &\geq \exp{\left(-\frac{\eta}{(1-\epsilon/4)}\, r_i\right)} \\
        &\geq \exp{\left(-\frac{\eta}{(1-\epsilon/4)}\, \frac{\log d}{(1+\epsilon/2)}\left(1 - \eta\, \exp{\left(-\frac{\kappa}{(1+\epsilon/2)}\right)}\right)^{i-1}\right)} \\
        &\geq \exp{\left(-\frac{\kappa}{(1+\epsilon/8)}\left(1 - \eta\, \exp{\left(-\frac{\kappa}{(1+\epsilon/2)}\right)}\right)^{i-1}\right)}.
    \end{align*}
    With this bound on $\keep_i$, we can get a lower bound on $\hat{\ell}_{i+1}$ as follows:
    \begin{align*}
        \hat{\ell}_{i+1} &= \hat{\ell}_1\,\prod\limits_{j \leq i}\keep_j \\
        &\geq \hat{\ell}_1\,\prod\limits_{j \leq i}\exp{\left(-\frac{\kappa}{(1+\epsilon/8)}\left(1 - \eta\, \exp{\left(-\frac{\kappa}{(1+\epsilon/2)}\right)}\right)^{j-1}\right)} \\
        &= (1+\epsilon)\,\frac{d}{\log d}\, \exp{\left(-\frac{\kappa}{(1+\epsilon/8)}\sum\limits_{j \leq i}\left(1 - \eta\, \exp{\left(-\frac{\kappa}{(1+\epsilon/2)}\right)}\right)^{j-1}\right)} \\
        &\geq (1+\epsilon)\,\frac{d}{\log d}\, \exp{\left(-\frac{\kappa}{(1+\epsilon/8)}\sum\limits_{j =1}^\infty\left(1 - \eta\, \exp{\left(-\frac{\kappa}{(1+\epsilon/2)}\right)}\right)^{j-1}\right)} \\
        &= (1+\epsilon)\,\frac{d}{\log d}\, \exp{\left(-\frac{\kappa}{(1+\epsilon/8)\eta}\, \exp{\left(\frac{\kappa}{(1+\epsilon/2)}\right)}\right)} \\
        &= (1+\epsilon)\,\frac{d}{\log d}\, \exp{\left(-\frac{\log d}{(1+\epsilon/8)}\, \exp{\left(\frac{\kappa}{(1+\epsilon/2)}\right)}\right)} \\
        &= (1+\epsilon)\,\frac{d}{\log d}\, d^{\left( \dfrac{-\exp{\left(\kappa/(1+\epsilon/2)\right)}}{(1+\epsilon/8)}\right)}.
    \end{align*}
    Recalling that $\kappa = (1+\epsilon/2)\log (1+\epsilon/100)$, we get
    $$\frac{\exp{\left(\kappa/(1+\epsilon/2)\right)}}{(1+\epsilon/8)} = \frac{1+\epsilon/100}{1+\epsilon/8} < 1-\epsilon/10.$$
    
    Therefore, for $d$ large enough, we get
    $$\hat{\ell}_{i+1} > (1+\epsilon)\,\frac{d}{\log d}\, d^{\epsilon/10-1} > d^{\epsilon/12}.$$
    Applying Lemma~\ref{lemma: error terms irrelevant}, we finally get the bound we desire:
    \[
        \ell_{i+1} \geq \hat{\ell}_{i+1}\left(1 - \frac{1}{\log\hat{\ell}_{i+1}}\right) \geq d^{\epsilon/12}\left(1 - \frac{1}{\log \hat{\ell}_{i+1}}\right) \geq d^{\epsilon/15}. \qedhere
    \]
\end{proof}

We can now finally establish the existence of the desired bound $i^\star$:

\begin{lemma}\label{lemma: i_star}
    If $d$ is sufficiently large, there exists an integer $i^\star \geq 1$ such that $\ell_{i^\star} \geq 8d_{i^\star}$.
\end{lemma}

\begin{proof}
    Let $r_i \coloneqq d_i/\ell_i$ and $\hat{r}_i \coloneqq \hat{d}_i/\hat{\ell}_i$ where $\hat{d}_i$ and $\hat{\ell}_i$ are defined as in Lemma~\ref{lemma: error terms irrelevant}. 
    Suppose, toward a contradiction, that $\ell_i < 8d_i$ (equivalently, $r_i > 1/8$) for all $i \geq 1$.  
    Note that $\hat{r}_i = \uncolor_i\, \hat{r}_{i-1}$ is a decreasing sequence. 
    Furthermore, as $d$ is sufficiently large, 
    \begin{equation}\label{eq:i_star_tech}
        \keep_j \geq \keep_1 = \left(1 - \frac{\eta}{\ell_1}\right)^{d_1} \ge 1 - \frac{\eta d_1}{\ell_1} = 1-\frac{\kappa}{1+\epsilon} \geq 1/2.
    \end{equation}
    Additionally, by Claim~\ref{claim: lb on hat ell and hat d} and Lemma~\ref{lemma:l_lower} we have $\hat{\ell}_i, \hat{d}_i \geq d^{\epsilon/20}$ for all $i$.
    Thus,
    \begin{align*}
        r_{i} &\leq \hat{r}_i \left( 1 + \frac{1}{\log\hat{d}_i}\right) \left( 1 - \frac{1}{\log\hat{\ell}_i}\right)^{-1} & \\
        &< 2\hat{r}_{i} & \\
        &\leq 2\hat{r}_1\prod\limits_{j < i}\left(1 - \eta\, \keep_j\right) & \\
        &\leq 2\hat{r}_1\left(1 - \frac{\eta}{2}\right)^i & \text{by \eqref{eq:i_star_tech}}\\
        &=  \frac{2\log d}{1 + \varepsilon}\left(1 - \frac{\eta}{2}\right)^i & \\
        & \leq 2\log d\, \exp{\left(-\frac{\eta}{2}\, i\right)}. &
    \end{align*}
    For $i \geq 10 \cdot \eta^{-1}\,\log\log d$, the last expression is less than $1/8$; a contradiction.
\end{proof}

We are now prepared to prove Theorem~\ref{theo: color-degree version of main result}. Let $i^\star \geq 1$ be the smallest integer such that $\ell_{i^\star} \geq 8d_{i^\star}$, which exists by Lemma~\ref{lemma: i_star}. 
Take any $i < i^\star$. 
We need to verify conditions \ref{item: d_i bound}--\ref{item: eta bound}. Note that Lemma~\ref{lemma:l_lower} yields for sufficiently large $d$,
\[ \ell_{i} \geq d^{\epsilon/15} \quad \text{and} \quad d_{i} \geq \frac{\ell_{i}}{8} \geq \frac{d^{\epsilon/15}}{8} \geq d^{\epsilon/20}. \]
Therefore, condition \ref{item: d_i bound} holds assuming that $d > \tilde{d}^{20/\epsilon}$. For \ref{item: ell_i bound}, we use Lemma \ref{lemma: decreasing ratios} to write
\[
    \frac{\ell_i}{d_i} \geq \frac{\ell_1}{d_1} > \frac{1}{\log d} > \frac{4\kappa}{\log d} =  4\eta,
\]
where the last inequality follows from $\varepsilon$ (and hence $\kappa$) being sufficiently small. 

Note that as we assume $i < i^\star$,
we have $\log \ell_i \geq \frac{\epsilon}{15} \log d$ and $\log d_i \geq \frac{\epsilon}{20} \log d$.
Therefore, $\eta \gg \log^{-3} \ell_i$ and $\eta \gg \log^{-3} d_i$, 
and we can apply
Lemmas~\ref{lemma: decreasing ratios} and~\ref{lemma: error terms irrelevant}.

Finally, as $d \ge d_i \ge d^{\varepsilon/20}$, it follows for sufficiently large $d$ that
\[
    \frac{1}{\log^2d_i} \leq \frac{1}{(\epsilon/20)^2\log^2 d} \leq \frac{\kappa}{\log d} = \eta < \frac{1}{4 \log d} \leq \frac{1}{4 \log d_i},
\]
so \ref{item: eta bound} holds as well. 
As discussed earlier, we can now iteratively apply Lemma~\ref{lem:nibble} $i^\star - 1$ times and then complete the coloring using Proposition~\ref{prop: final blow}. This completes the proof of Theorem~\ref{theo: color-degree version of main result}.

\section*{Acknowledgments}
We are grateful to Luke Postle for bringing \cite{delcourt2022finding} to our attention. We also thank Eoin Hurley and Yuval Wigderson for helpful discussions.

\printbibliography

@article{LiPostle,
  title={The chromatic number of triangle-free hypergraphs},
  author={Li, Lina and Postle, Luke},
  journal={arXiv preprint arXiv:2202.02839},
  year={2022}
}

@article{ajtai1981dense,
  title={A dense infinite {Sidon} sequence},
  author= {Ajtai, Mikl{\'{o}}s and Koml{\'{o}}s, J{\'{a}}nos and Szemer{\'{e}}di, Endre},
  fjournal={European Journal of Combinatorics},
  journal= {European J. Combin.},
  volume={2},
  number={1},
  pages={1--11},
  year={1981},
  publisher={Academic Press}
}

@article{KPS82,
  title={A lower bound for {H}eilbronn's problem},
  author={Koml{\'o}s, J{\'a}nos and Pintz, J{\'a}nos and Szemer{\'e}di, Endre},
  fjournal={Journal of the London Mathematical Society},
  journal = {J. Lond. Math. Soc.},
  volume={2},
  number={1},
  pages={13--24},
  year={1982},
  publisher={Wiley Online Library}
}

@article {AKS1980,
    AUTHOR = {Ajtai, Mikl{\'{o}}s and Koml{\'{o}}s, J{\'{a}}nos and Szemer{\'{e}}di, Endre},
     TITLE = {A note on {R}amsey numbers},
   JOURNAL = {J. Combin. Theory Ser. A},
  FJOURNAL = {Journal of Combinatorial Theory. Series A},
    VOLUME = {29},
      YEAR = {1980},
    NUMBER = {3},
     PAGES = {354--360},
      ISSN = {0097-3165},
   OPTMRCLASS = {05C55 (05C35)},
  OPToptMRNUMBER = {600598},
OPTMRREVIEWER = {J. E. Graver},
       DOI = {10.1016/0097-3165(80)90030-8},
       URL = {https://doi.org/10.1016/0097-3165(80)90030-8},
}

@article{Sh83,
  title={A note on the independence number of triangle-free graphs},
  author={Shearer, James B},
  fjournal={Discrete Mathematics},
  journal = {Discrete Math.},
  volume={46},
  number={1},
  pages={83--87},
  year={1983},
  publisher={Elsevier}
}

@article{Sh91,
  title={A note on the independence number of triangle-free graphs, {II}},
  author={Shearer, James B},
  fjournal={Journal of Combinatorial Theory, Series B},
  journal={J. Combin. Theory Ser. B},
  volume={53},
  number={2},
  pages={300--307},
  year={1991},
  publisher={Elsevier}
}

@article{AEKS81,
  title={On {T}ur{\'a}n’s theorem for sparse graphs},
  author={Ajtai, Mikl{\'o}s and Erd{\H{o}}s, Paul and Koml{\'o}s, J{\'a}nos and Szemer{\'e}di, Endre},
  journal={Combinatorica},
  volume={1},
  number={4},
  pages={313--317},
  year={1981},
  publisher={Springer}
}

@article{alon1996independence,
  title={Independence numbers of locally sparse graphs and a {R}amsey type problem},
  author={Alon, Noga},
  fjournal={Random Structures \& Algorithms},
  journal   = {Random Structures Algorithms},
  volume={9},
  number={3},
  pages={271--278},
  year={1996},
  publisher={Wiley Online Library}
}

@article {Alon3,
    AUTHOR = {Alon, Noga},
     TITLE = {The strong chromatic number of a graph},
   JOURNAL = {Random Structures Algorithms},
  FJOURNAL = {Random Structures \& Algorithms},
    VOLUME = {3},
      YEAR = {1992},
    NUMBER = {1},
     PAGES = {1--7},
      ISSN = {1042-9832,1098-2418},
   MRCLASS = {05C15},
  MRNUMBER = {1139484},
MRREVIEWER = {Ruth\ Bari},
       DOI = {10.1002/rsa.3240030102},
       URL = {https://doi.org/10.1002/rsa.3240030102},
}

@article {AlonChoosability,
    AUTHOR = {Alon, Noga},
     TITLE = {Choice numbers of graphs: a probabilistic approach},
   JOURNAL = {Combin. Probab. Comput.},
  FJOURNAL = {Combinatorics, Probability and Computing},
    VOLUME = {1},
      YEAR = {1992},
    NUMBER = {2},
     PAGES = {107--114},
      ISSN = {0963-5483,1469-2163},
   MRCLASS = {05C80 (05C15)},
  MRNUMBER = {1179241},
MRREVIEWER = {J.\ Spencer},
       DOI = {10.1017/S0963548300000122},
       URL = {https://doi.org/10.1017/S0963548300000122},
}

@ARTICLE{BJ15,
    author = {Bruhn, Henning and Joos, Felix},
     TITLE = {A stronger bound for the strong chromatic index},
   JOURNAL = {Combin. Probab. Comput.},
  FJOURNAL = {Combinatorics, Probability and Computing},
    VOLUME = {27},
      YEAR = {2018},
    NUMBER = {1},
     PAGES = {21--43},
      ISSN = {0963-5483},
OPTMRCLASS = {05C15},
 OPToptMRNUMBER = {3734328},
       URL = {https://doi.org/10.1017/S0963548317000244},
}

@ARTICLE{BPP18,
       author={Bonamy, Marthe and Perrett, Thomas and Postle, Luke},
        title = {Colouring Graphs with Sparse Neighbourhoods: Bounds and Applications},
      journal = {arXiv preprint arXiv:1810.06704},
         year = 2018,
        month = Oct,
          eid = {arXiv:1810.06704},
archivePrefix = {arXiv},
       eprint = {1810.06704},
 primaryClass = {math.CO},
       adsurl = {https://ui.adsabs.harvard.edu/\#abs/2018arXiv181006704B}
}

@inproceedings {HdVK20,
    AUTHOR = {Hurley, Eoin and de Joannis de Verclos, R\'{e}mi and Kang, Ross
              J.},
     TITLE = {An improved procedure for colouring graphs of bounded local
              density},
 BOOKTITLE = {Proceedings of the 2021 {ACM}-{SIAM} {S}ymposium on {D}iscrete
              {A}lgorithms ({SODA})},
     PAGES = {135--148},
 PUBLISHER = {[Society for Industrial and Applied Mathematics (SIAM)],
              Philadelphia, PA},
      YEAR = {2021},
   MRCLASS = {05C15},
  MRNUMBER = {4262443},
       DOI = {10.1137/1.9781611976465.10},
       URL = {https://doi.org/10.1137/1.9781611976465.10},
       }

@incollection{EN85,
  author = {Erd{\H{o}}s, Paul and Ne{\v{s}}et{\v{r}}il, Jaroslav},
  booktitle={Irregularities of partitions},
  editor={Hal{\'a}sz, G{\'a}bor and S{\'o}s, Vera T},
  year={1989},
  chapter={Problems},
  pages={162--163},
  publisher={Springer, Berlin, Heidelberg},
  series={Algorithms and Combinatorics 8},
}

@article {Vu02,
    AUTHOR = {Vu, Van H.},
     TITLE = {A general upper bound on the list chromatic number of locally
              sparse graphs},
   JOURNAL = {Combin. Probab. Comput.},
  FJOURNAL = {Combinatorics, Probability and Computing},
    VOLUME = {11},
      YEAR = {2002},
    NUMBER = {1},
     PAGES = {103--111},
      ISSN = {0963-5483},
  OPTMRCLASS = {05C15 (05C35)},
  OPToptMRNUMBER = {1888186},
OPTMRREVIEWER = {Andr{\'{a}}s Gy{\'{a}}rf{\'{a}}s},
       DOI = {10.1017/S0963548301004898},
       URL = {https://doi.org/10.1017/S0963548301004898},
}

@article{vu1999some,
  title={On some simple degree conditions that guarantee the upper bound on the chromatic (choice) number of random graphs},
  author={Vu, Van H},
  journal={Journal of Graph Theory},
  volume={31},
  number={3},
  pages={201--226},
  year={1999},
  publisher={Wiley Online Library}
}

@incollection{kahn1997,
  title={On some hypergraph problems of {P}aul {E}rd{\H{o}}s and the asymptotics of matchings, covers and colorings},
  author={Kahn, Jeff},
  booktitle={The Mathematics of {P}aul {E}rd{\"o}s I},
  pages={345--371},
  year={1997},
  publisher={Springer}
}

@article {erdos1981,
    AUTHOR = {Erd\H{o}s, Paul},
     TITLE = {On the combinatorial problems which {I} would most like to see solved},
   JOURNAL = {Combinatorica},
  FJOURNAL = {Combinatorica. An International Journal of the J\'anos Bolyai
              Mathematical Society},
    VOLUME = {1},
      YEAR = {1981},
    NUMBER = {1},
     PAGES = {25--42},
      ISSN = {0209-9683},
   OPTMRCLASS = {05-02 (04A20)},
 OPToptMRNUMBER = {602413},
OPTMRREVIEWER = {L. C. Eggan},
       DOI = {10.1007/BF02579174},
       URL = {https://doi.org/10.1007/BF02579174},
}

@article{kelly2024special,
  title={A special case of Vu’s conjecture: colouring nearly disjoint graphs of bounded maximum degree},
  author={Kelly, Tom and K{\"u}hn, Daniela and Osthus, Deryk},
  journal={Combinatorics, Probability and Computing},
  volume={33},
  number={2},
  pages={179--195},
  year={2024},
  publisher={Cambridge University Press}
}

@article{EFL2023,
  author    = {Dong Yeap Kang and Tom Kelly and Daniela K{\"u}hn and Abhishek Methuku and Deryk Osthus},
  title     = {A Proof of the Erd{\H{o}}s--Faber--Lov{\'a}sz Conjecture},
  journal   = {Annals of Mathematics},
  volume    = {198},
  number    = {2},
  pages     = {451--486},
  year      = {2023},
  doi       = {10.4007/annals.2023.198.2.2}
}

@incollection{KangKelly2023nibble,
  author    = {Dong Yeap Kang and Tom Kelly and Daniela Kühn and Abhishek Methuku and Deryk Osthus},
  title     = {Graph and hypergraph colouring via nibble methods: A survey},
  booktitle = {Proceedings of the 8th European Congress of Mathematics},
  pages     = {771--823},
  year      = {2023},
  publisher = {EMS Press},
  doi       = {10.4171/8ECM/11}
}

@article {kahn1996asymptotically,
    AUTHOR = {Kahn, Jeff},
     TITLE = {Asymptotically good list-colorings},
   JOURNAL = {J. Combin. Theory Ser. A},
  FJOURNAL = {Journal of Combinatorial Theory. Series A},
    VOLUME = {73},
      YEAR = {1996},
    NUMBER = {1},
     PAGES = {1--59},
      ISSN = {0097-3165},
  OPTMRCLASS = {05-02 (05C15 05C65)},
  OPToptMRNUMBER = {1367606},
OPTMRREVIEWER = {Hugh Hind},
       DOI = {10.1006/jcta.1996.0001},
       URL = {https://doi.org/10.1006/jcta.1996.0001},
}

@article {R98,
    AUTHOR = {Bruce Reed},
     TITLE = {{$\omega,\ \Delta$}, and {$\chi$}},
   JOURNAL = {J. Graph Theory},
  FJOURNAL = {Journal of Graph Theory},
    VOLUME = {27},
      YEAR = {1998},
    NUMBER = {4},
     PAGES = {177--212},
      ISSN = {0364-9024},
   MRCLASS = {05C15 (05C80)},
  optMRNUMBER = {1610746},
MRREVIEWER = {Ioan Tomescu},
       DOI = {10.1002/(SICI)1097-0118(199804)27:4<177::AID-JGT1>3.0.CO;2-K},
       URL =
              {https://doi.org/10.1002/(SICI)1097-0118(199804)27:4<177::AID-JGT1>3.0.CO;2-K},
}

@unpublished{J96-Kr,
author = {Johansson, Anders},
title = {The choice number of sparse graphs},
note = {Unpublished Manuscript},
year = {1996}
}

@article {AKS,
    AUTHOR = {Alon, Noga and Krivelevich, Michael and Sudakov, Benny},
     TITLE = {Coloring graphs with sparse neighborhoods},
   JOURNAL = {J. Combin. Theory Ser. B},
  FJOURNAL = {Journal of Combinatorial Theory. Series B},
    VOLUME = {77},
      YEAR = {1999},
    NUMBER = {1},
     PAGES = {73--82},
      ISSN = {0095-8956,1096-0902},
   MRCLASS = {05C15},
  MRNUMBER = {1710532},
MRREVIEWER = {David\ E.\ Woolbright},
       DOI = {10.1006/jctb.1999.1910},
       URL = {https://doi.org/10.1006/jctb.1999.1910},
}

@article{Sh95,
    AUTHOR = {Shearer, James B},
     TITLE = {On the independence number of sparse graphs},
   JOURNAL = {Random Structures Algorithms},
  FJOURNAL = {Random Structures \& Algorithms},
    VOLUME = {7},
      YEAR = {1995},
    NUMBER = {3},
     PAGES = {269--271},
      ISSN = {1042-9832},
   MRCLASS = {05C35},
  optMRNUMBER = {1369066},
MRREVIEWER = {E. M. Palmer},
       URL = {https://doi.org/10.1002/rsa.3240070305},
}

@techreport{johansson1996,
  title={Asymptotic choice number for triangle free graphs},
  author={Johansson, Anders},
  year={1996},
  institution={DIMACS technical report}
}

@article{M17,
  author    = {Molloy, Michael},
  title     = {The list chromatic number of graphs with small clique number},
  journal   = {J. Combin. Theory Ser. B},
  volume    = {134},
  pages     = {264--284},
  year      = {2019},
  url       = {https://doi.org/10.1016/j.jctb.2018.06.007},
  doi       = {10.1016/j.jctb.2018.06.007},
  bibsource = {dblp computer science bibliography, https://dblp.org}
}

@article {kim1995,
    AUTHOR = {Kim, Jeong Han},
     TITLE = {On {B}rooks' theorem for sparse graphs},
   JOURNAL = {Combin. Probab. Comput.},
  FJOURNAL = {Combinatorics, Probability and Computing},
    VOLUME = {4},
      YEAR = {1995},
    NUMBER = {2},
     PAGES = {97--132},
      ISSN = {0963-5483},
   OPTMRCLASS = {05C15 (05C35)},
  OPToptMRNUMBER = {1342856},
OPTMRREVIEWER = {A. G. Thomason},
       DOI = {10.1017/S0963548300001528},
       URL = {https://doi.org/10.1017/S0963548300001528},
}

@inproceedings{KovariSosTuran,
  title={On a problem of Zarankiewicz},
  author={K{\H{o}}v{\'a}ri, P and T S{\'o}s, Vera and Tur{\'a}n, P{\'a}l},
  booktitle={Colloquium Mathematicum},
  volume={3},
  pages={50--57},
  year={1954},
  organization={Polska Akademia Nauk}
}

@article{AndersonBernshteynDhawan,
  title={Coloring graphs with forbidden bipartite subgraphs},
  author={Anderson, James and Bernshteyn, Anton and Dhawan, Abhishek},
  journal={Combinatorics, Probability and Computing},
  volume={32},
  number={1},
  pages={45--67},
  year={2023},
  publisher={Cambridge University Press}
}

@book{MolloyReed,
  title={Graph colouring and the probabilistic method},
  author={Molloy, Michael and Reed, Bruce},
  volume={23},
  year={2002},
  publisher={Springer Science \& Business Media}
}

@article{Mahdian2000strong,
  title={The strong chromatic index of C4-free graphs},
  author={Mahdian, Mohammad},
  journal={Random Structures \& Algorithms},
  volume={17},
  number={3-4},
  pages={357--375},
  year={2000},
  publisher={Wiley Online Library}
}

@article{bruhn2018stronger,
  title={A stronger bound for the strong chromatic index},
  author={Bruhn, Henning and Joos, Felix},
  journal={Combinatorics, Probability and Computing},
  volume={27},
  number={1},
  pages={21--43},
  year={2018},
  publisher={Cambridge University Press}
}

@book {Mitzenmacher,
    AUTHOR = {Mitzenmacher, Michael and Upfal, Eli},
     TITLE = {Probability and computing},
   EDITION = {Second},
      NOTE = {Randomization and probabilistic techniques in algorithms and
              data analysis},
 PUBLISHER = {Cambridge University Press, Cambridge},
      YEAR = {2017},
     PAGES = {xx+467},
      ISBN = {978-1-107-15488-9},
   MRCLASS = {68-01 (60C05 60G42 60J10 60K25 62H30 68W20 68W40)},
  MRNUMBER = {3674428},
}

@article{dhawan2024palette,
  title={Palette Sparsification for Graphs with Sparse Neighborhoods},
  author={Dhawan, Abhishek},
  journal={arXiv preprint arXiv:2408.08256},
  year={2024}
}

@article{Sparsification1,
	author = {S.K. Bera and A. Chakrabarti and P. Ghosh},
	title = {Graph coloring via degeneracy in streaming and other space-conscious models},
	journaltitle = {International Colloquium on Automata, Languages, and Programming (ICALP)},
	date = {2020},
	pages = {\#11},
	addendum = {Full version: \url{https://arxiv.org/abs/1905.00566}},
}

@article{Sparsification2,
	author = {M.M. Halld{\'{o}}rsson and F. Kuhn and A. Nolin and T. Tonoyan},
	title = {Near-optimal distributed degree+1 coloring},
	journaltitle = {ACM SIGACT Symposium on Theory of Computing (STOC)},
	date = {2022},
	pages = {450--463},
	addendum = {Full version: \url{https://arxiv.org/abs/2112.00604}},
}

@article{Sparsification3,
	author = {S. Assadi and P. Kumar and P. Mittal},
	title = {Brooks' theorem in graph streams: A single-pass semi-streaming algorithm for $\Delta$-coloring},
	journaltitle = {ACM SIGACT Symposium on Theory of Computing (STOC)},
	date = {2022},
	pages = {234--247},
	addendum = {Full version: \url{https://arxiv.org/abs/2203.10984}},
}

@article{alon2020palette,
    author = {N. Alon and S. Assadi},
    title = {Palette sparsification beyond $(\Delta + 1)$ vertex coloring},
    journaltitle = {Approximation, Randomization, and Combinatorial Optimization. Algorithms and Techniques (APPROX/RANDOM)},
    date = {2020},
    pages = {\#6},
    addendum = {Full version: \url{https://arxiv.org/abs/2006.10456}},
}

@article{anderson2025coloring,
  title={Coloring graphs with forbidden almost bipartite subgraphs},
  author={Anderson, James and Bernshteyn, Anton and Dhawan, Abhishek},
  journal={Random Structures \& Algorithms},
  volume={66},
  number={4},
  pages={e70012},
  year={2025},
  publisher={Wiley Online Library}
}

@article{MT,
	author = {R. Moser and G. Tardos},
	title = {A constructive proof of the general Lov\'{a}sz Local Lemma},
	journaltitle = {J. ACM},
	date = {2010},
	volume = {57},
	number = {2},
}

@inproceedings{assadi2019sublinear,
  title={Sublinear algorithms for ($\Delta$+ 1) vertex coloring},
  author={Assadi, Sepehr and Chen, Yu and Khanna, Sanjeev},
  booktitle={Proceedings of the Thirtieth Annual ACM-SIAM Symposium on Discrete Algorithms},
  pages={767--786},
  year={2019},
  organization={SIAM}
}

@article{reed1999list,
  title={The list colouring constants},
  author={Reed, Bruce},
  journal={Journal of Graph Theory},
  volume={31},
  number={2},
  pages={149--153},
  year={1999},
  publisher={John Wiley \& Sons, Inc. New York, NY, USA}
}

@article{reed2002asymptotically,
  title={Asymptotically the list colouring constants are 1},
  author={Reed, Bruce and Sudakov, Benny},
  journal={Journal of Combinatorial Theory, Series B},
  volume={86},
  number={1},
  pages={27--37},
  year={2002},
  publisher={Elsevier}
}

@article{Linial,
	author = {N. Linial},
	title = {Locality in distributed graph algorithms},
	journaltitle = {SIAM J. Comput.},
	volume = {21},
	number = {1},
	pages = {193--201},
	date = {1992}
}

@article{grable2000fast,
  title={Fast distributed algorithms for brooks--vizing colorings},
  author={Grable, David A and Panconesi, Alessandro},
  journal={Journal of Algorithms},
  volume={37},
  number={1},
  pages={85--120},
  year={2000},
  publisher={Elsevier}
}

@inproceedings{PS15,
  title={Fast distributed coloring algorithms for triangle-free graphs},
  author={Pettie, Seth and Su, Hsin-Hao},
  booktitle={International Colloquium on Automata, Languages, and Programming},
  pages={681--693},
  year={2013},
  organization={Springer}
}

@inproceedings{bhattacharya2021online,
  title={Online edge coloring algorithms via the nibble method},
  author={Bhattacharya, Sayan and Grandoni, Fabrizio and Wajc, David},
  booktitle={Proceedings of the 2021 ACM-SIAM Symposium on Discrete Algorithms (SODA)},
  pages={2830--2842},
  year={2021},
  organization={SIAM}
}

@inproceedings{bhattacharya2024nibbling,
  title={Nibbling at long cycles: Dynamic (and static) edge coloring in optimal time},
  author={Bhattacharya, Sayan and Costa, Mart{\'\i}n and Panski, Nadav and Solomon, Shay},
  booktitle={Proceedings of the 2024 Annual ACM-SIAM Symposium on Discrete Algorithms (SODA)},
  pages={3393--3440},
  year={2024},
  organization={SIAM}
}

@inproceedings{chung2014distributed,
  title={Distributed algorithms for the Lov{\'a}sz local lemma and graph coloring},
  author={Chung, Kai-Min and Pettie, Seth and Su, Hsin-Hao},
  booktitle={Proceedings of the 2014 ACM symposium on Principles of distributed computing},
  pages={134--143},
  year={2014}
}

@article{davies2020graph,
  title={Graph structure via local occupancy},
  author={Davies, Ewan and Kang, Ross J and Pirot, Fran{\c{c}}ois and Sereni, Jean-S{\'e}bastien},
  journal={arXiv preprint arXiv:2003.14361},
  year={2020}
}

@article{dhawan2025bounds,
  title={Bounds for the Independence and Chromatic Numbers of Locally Sparse Graphs},
  author={Dhawan, Abhishek},
  journal={Annals of Combinatorics},
  pages={1--28},
  year={2025},
  publisher={Springer}
}

@article{alon2021asymmetric,
  title={Asymmetric list sizes in bipartite graphs},
  author={Alon, Noga and Cambie, Stijn and Kang, Ross J},
  journal={Annals of Combinatorics},
  volume={25},
  number={4},
  pages={913--933},
  year={2021},
  publisher={Springer}
}

@article{campos2023new,
  title={A new lower bound for sphere packing},
  author={Campos, Marcelo and Jenssen, Matthew and Michelen, Marcus and Sahasrabudhe, Julian},
  journal={arXiv preprint arXiv:2312.10026},
  year={2023}
}

@article{bernshteyn2019johansson,
  title={The Johansson-Molloy theorem for DP-coloring},
  author={Bernshteyn, Anton},
  journal={Random Structures \& Algorithms},
  volume={54},
  number={4},
  pages={653--664},
  year={2019},
  publisher={Wiley Online Library}
}

@article{hurley2021first,
  title={A first moment proof of the Johansson-Molloy theorem},
  author={Hurley, EOIN and Pirot, FRAN{\c{C}}OIS},
  journal={arXiv preprint arXiv:2109.15215},
  year={2021}
}

@article{bonamy2022bounding,
  title={Bounding $\chi$ by a fraction of $\Delta$ for graphs without large cliques},
  author={Bonamy, Marthe and Kelly, Tom and Nelson, Peter and Postle, Luke},
  journal={Journal of Combinatorial Theory, Series B},
  volume={157},
  pages={263--282},
  year={2022},
  publisher={Elsevier}
}

@article{martinsson2021simplified,
  title={A simplified proof of the Johansson-Molloy Theorem using the Rosenfeld counting method},
  author={Martinsson, Anders},
  journal={arXiv preprint arXiv:2111.06214},
  year={2021}
}

@article{anderson2024coloring,
  title={Coloring locally sparse graphs},
  author={Anderson, James and Dhawan, Abhishek and Kuchukova, Aiya},
  journal={arXiv preprint arXiv:2402.19271},
  year={2024}
}

@article{dhawan2023list,
  title={List colorings of $ k $-partite $ k $-graphs},
  author={Dhawan, Abhishek},
  journal = {Electronic Journal of Combinatorics},
  volume={32},
  number={2},
  pages={2.16},
  year={2025},
  DOI = {10.37236/12657},
  URL = {https://doi.org/10.37236/12657},
}

@article{cambie2022independent,
  title={Independent transversals in bipartite correspondence-covers},
  author={Cambie, Stijn and Kang, Ross J},
  journal={Canadian Mathematical Bulletin},
  volume={65},
  number={4},
  pages={882--894},
  year={2022},
  publisher={Canadian Mathematical Society}
}

@article{delcourt2022finding,
  title={Finding an almost perfect matching in a hypergraph avoiding forbidden submatchings},
  author={Delcourt, Michelle and Postle, Luke},
  journal={arXiv preprint arXiv:2204.08981},
  year={2022}
}

\end{document}